\documentclass{amsart}

\usepackage{enumerate, amsmath, amsfonts, amssymb, amsthm, wasysym, graphics, graphicx, xcolor, url, hyperref, hypcap, a4wide, xargs, overpic}
\hypersetup{colorlinks=true, citecolor=darkblue, linkcolor=darkblue}
\usepackage[all]{xy}
\usepackage{tikz}\usetikzlibrary{trees,snakes,shapes,arrows,matrix,calc}
\graphicspath{{figures/}}

\title{Graph properties of graph associahedra}

\thanks{VP was partially supported by grant MTM2011-22792 of the Spanish MICINN and by the French ANR grant EGOS (12 JS02 002 01).}

\author{Thibault Manneville}
\address{LIX, \'Ecole Polytechnique, Palaiseau}
\email{thibault.manneville@lix.polytechnique.fr}
\urladdr{http://www.lix.polytechnique.fr/~manneville/}

\author{Vincent Pilaud}
\address{CNRS \& LIX, \'Ecole Polytechnique, Palaiseau}
\email{vincent.pilaud@lix.polytechnique.fr}
\urladdr{http://www.lix.polytechnique.fr/~pilaud/}

\newtheorem{theorem}{Theorem}
\newtheorem{corollary}[theorem]{Corollary}
\newtheorem{proposition}[theorem]{Proposition}
\newtheorem{lemma}[theorem]{Lemma}

\theoremstyle{definition}
\newtheorem{example}[theorem]{Example}
\newtheorem{remark}[theorem]{Remark}
\newtheorem{question}[theorem]{Question}

\newcommand{\R}{\mathbb{R}} 
\newcommand{\fS}{\mathfrak{S}} 

\newcommand{\set}[2]{\left\{ #1 \;\middle|\; #2 \right\}} 
\newcommand{\ssm}{\smallsetminus} 
\newcommand{\symdif}{\, \triangle \, } 
\newcommand{\eqdef}{\mbox{\,\raisebox{0.2ex}{\scriptsize\ensuremath{\mathrm:}}\ensuremath{=}\,}} 

\newcommand{\Asso}{\mathsf{Asso}} 
\newcommand{\Perm}{\mathsf{Perm}} 
\newcommand{\Nest}{\mathsf{Nest}} 
\newcommand{\polytope}{P} 
\newcommand{\face}{F} 

\newcommand{\ground}{\mathrm{V}} 
\newcommandx{\graphG}[1][1=G]{\mathrm{#1}} 
\newcommand{\pathG}{\graphG[P]} 
\newcommand{\cycleG}{\graphG[O]} 
\newcommand{\completeG}{\graphG[K]} 
\newcommand{\starG}{\graphG[X]} 
\newcommand{\leave}{l} 
\newcommand{\connectedComponents}[1]{\pi_0(#1)} 

\newcommandx{\block}[1][1=b]{\mathsf{#1}} 
\newcommandx{\building}[1][1=B]{\mathsf{#1}} 
\newcommand{\nestelem}{\block[n]} 
\newcommand{\nested}{\building[N]} 
\newcommand{\tube}{\block[t]} 
\newcommand{\tubing}{\building[T]} 
\newcommand{\linearExtensionTubing}{\building[L]} 
\newcommandx{\snode}[1][1=s]{\mathsf{#1}} 
\newcommandx{\spine}[1][1=S]{\mathsf{#1}} 
\newcommand{\lab}{\lambda} 
\newcommand{\nestedComplex}{\mathcal{N}} 
\newcommand{\flipGraph}{\mathcal{F}} 
\newcommand{\flip}{\leftrightarrow} 

\newcommand{\diam}[1]{\delta(\flipGraph(#1))} 
\newcommand{\surjection}{\Omega} 
\newcommand{\normalization}{\surjection^\star} 
\newcommand{\bridge}{\mathsf{B}} 
\newcommandx{\shortFlip}[1][1=f]{\mathbf{#1}} 
\newcommand{\longFlip}{\mathbf{g}} 
\newcommand{\descendants}{\mathrm{desc}} 
\newcommand{\coarsens}{\prec} 
\newcommand{\HamiltonianCycle}{\mathcal{H}} 
\newcommand{\tail}[1]{\mathring{#1}} 

\DeclareMathOperator{\conv}{conv} 

\newcommand{\fref}[1]{Figure~\ref{#1}} 
\newcommand{\ie}{\textit{i.e.}~} 
\newcommand{\eg}{\textit{e.g.}~} 
\definecolor{darkblue}{rgb}{0,0,0.7} 
\newcommand{\darkblue}{\color{darkblue}} 
\newcommand{\defn}[1]{\emph{\darkblue #1}} 
\newcommand{\para}[1]{\medskip\paragraph{\bf #1}} 
\usepackage{todonotes}

\begin{document}

\begin{abstract}
A graph associahedron is a simple polytope whose face lattice encodes the nested structure of the connected subgraphs of a given graph. In this paper, we study certain graph properties of the $1$-skeleta of graph associahedra, such as their diameter and their Hamiltonicity. Our results extend known results for the classical associahedra (path associahedra) and permutahedra (complete graph associahedra). We also discuss partial extensions to the family of nestohedra.

\medskip
\noindent
{\sc keywords.}
Graph associahedron, nestohedron, graph diameter, Hamiltonian cycle.
\end{abstract}

\maketitle


\section{Introduction}
\label{sec:intro}

\defn{Associahedra} are classical polytopes whose combinatorial structure was first investigated by \mbox{J.~Stasheff~\cite{Stasheff}} and later geometrically realized by several methods~\cite{Lee, GelfandKapranovZelevinsky, Loday, HohlwegLange, PilaudSantos-brickPolytope, CeballosSantosZiegler}. They appear in different contexts in mathematics, in particular in algebraic combinatorics (in homotopy theory~\cite{Stasheff}, for construction of Hopf algebras~\cite{LodayRonco}, in cluster algebras~\cite{ChapotonFominZelevinsky, HohlwegLangeThomas}) and discrete geometry (as instances of secondary or fiber polytopes~\cite{GelfandKapranovZelevinsky, BilleraFillimanSturmfels} or brick polytopes~\cite{PilaudSantos-brickPolytope, PilaudStump-brickPolytope}). The combinatorial structure of the $n$-dimensional associahedron encodes the dissections of a convex $(n+3)$-gon: its vertices correspond to the triangulations of the $(n+3)$-gon, its edges correspond to flips between them, etc. See \fref{fig:exmAssociahedra}. Various combinatorial properties of these polytopes have been studied, in particular in connection with the symmetric group and the \defn{permutahedron}. The combinatorial structure of the $n$-dimensional permutahedron encodes ordered partitions of~$[n+1]$: its vertices are the permutations of~$[n+1]$, its edges correspond to transpositions of adjacent letters, etc.

In this paper, we are interested in graph properties, namely in the diameter and Hamiltonicity, of the $1$-skeleta of certain generalizations of the permutahedra and the associahedra. For the \mbox{$n$-dimensional} permutahedron, the diameter of the \defn{transposition graph} is the number~$\binom{n+1}{2}$ of inversions of the longest permutation of~$[n+1]$. Moreover, H.~Steinhaus~\cite{Steinhaus}, S.~M.~Johnson~\cite{Johnson}, and H.~F.~Trotter~\cite{Trotter} independently designed an algorithm to construct a Hamiltonian cycle of this graph. For the associahedron, the diameter of the \defn{flip graph} motivated intensive research and relevant approaches, involving volumetric arguments in hyperbolic geometry~\cite{SleatorTarjanThurston} and combinatorial properties of Thompson's groups~\cite{Dehornoy}. Recently, L.~Pournin finally gave a purely combinatorial proof that the diameter of the $n$-dimensional associahedron is precisely~$2n-4$ as soon as~$n \ge 9$~\cite{Pournin}. On the other hand, J.~Lucas~\cite{Lucas} proved that the flip graph is Hamiltonian. Later, F.~Hurtado and M.~Noy~\cite{HurtadoNoy} obtained a simpler proof of this result, using a hierarchy of triangulations which organizes all triangulations of convex polygons into an infinite generating~tree. 

Generalizing the classical associahedron, M.~Carr and S.~Devadoss~\cite{CarrDevadoss, Devadoss} defined and constructed \defn{graph associahedra}. For a finite graph~$\graphG$, a \defn{$\graphG$-associahedron}~$\Asso(\graphG)$ is a simple convex polytope whose combinatorial structure encodes the connected subgraphs of~$\graphG$ and their nested structure. To be more precise, the face lattice of the polar of a $\graphG$-associahedron is isomorphic to the \defn{nested complex} on~$\graphG$, defined as the simplicial complex of all collections of tubes (vertex subsets inducing connected subgraphs) of~$\graphG$ which are pairwise either nested, or disjoint and non-adjacent. See Figures~\ref{fig:exmAssociahedra} and~\ref{fig:classicAssociahedra3} for $3$-dimensional examples. The graph associahedra of certain special families of graphs happen to coincide with well-known families of polytopes (see \fref{fig:classicAssociahedra3}): classical associahedra are path associahedra, cyclohedra are cycle associahedra, and permutahedra are complete graph associahedra. Graph associahedra have been geometrically realized in different ways: by successive truncations of faces of the standard simplex~\cite{CarrDevadoss}, as Minkowski sums of faces of the standard simplex~\cite{Postnikov, FeichtnerSturmfels}, or from their normal fans by exhibiting explicit inequality descriptions~\cite{Zelevinsky}. However, we do not consider these geometric realizations as we focus on the combinatorial properties of the nested complex.

\begin{figure}
  \capstart
  \centerline{\includegraphics[width=\textwidth]{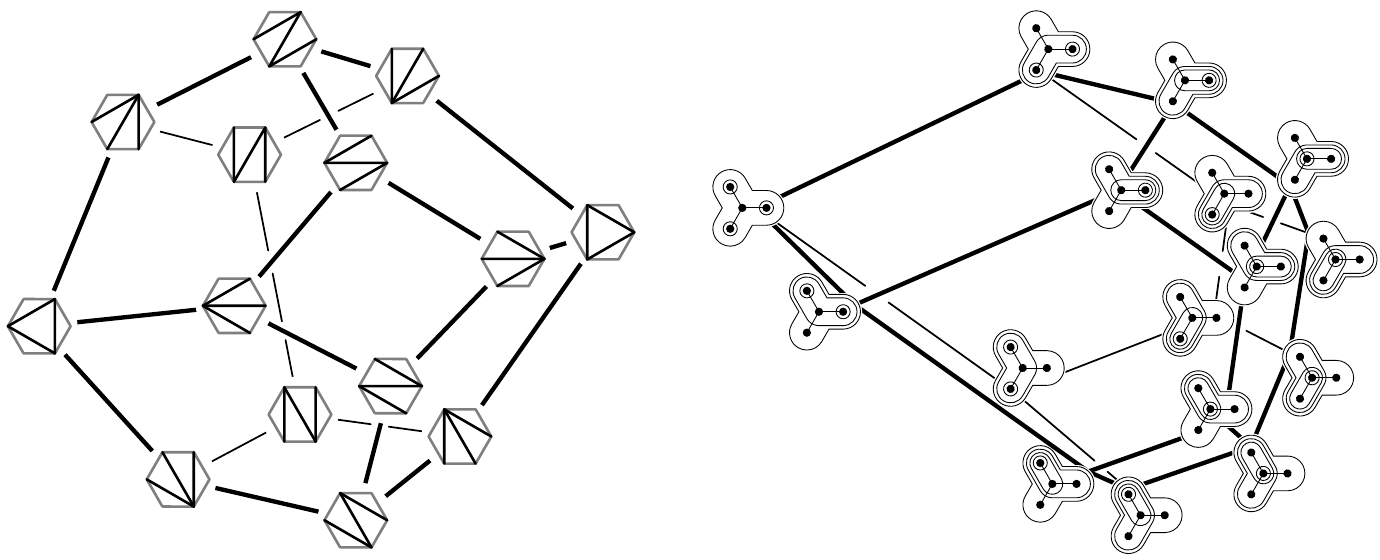}}
  \caption{The $3$-dimensional associahedron and the graph associahedron of the tripod.}
  \label{fig:exmAssociahedra}
\end{figure}

Given a finite simple graph~$\graphG$, we denote by~$\flipGraph(\graphG)$ the $1$-skeleton of the graph associahedron~$\Asso(\graphG)$. In other words, $\flipGraph(\graphG)$ is the facet-ridge graph of the nested complex on~$\graphG$. Its vertices are maximal tubings on~$\graphG$ and its edges connect tubings which differ only by two tubes. See Section~\ref{sec:preliminaries} for precise definitions and examples. In this paper, we study graph properties of~$\flipGraph(\graphG)$. In Section~\ref{sec:diameter}, we focus on the diameter~$\diam{\graphG}$ of the flip graph~$\flipGraph(\graphG)$. We obtain the following structural results.

\begin{theorem}
\label{theo:increasingDiameter}
The diameter~$\diam{\graphG}$ of the flip graph~$\flipGraph(\graphG)$ is non-decreasing: $\diam{\graphG} \le \diam{\graphG'}$ for any two graphs~$\graphG, \graphG'$ such that~$\graphG \subseteq \graphG'$.
\end{theorem}

Related to this diameter, we investigate the \defn{non-leaving-face property}: do all geodesics between two vertices of a face~$\face$ of~$\Asso(\graphG)$ stay in~$\face$? This property was proved for the classical associahedron in~\cite{SleatorTarjanThurston} but the name was coined in~\cite{CeballosPilaud-diameter}. Although not all faces of the graph associahedron~$\Asso(\graphG)$ fulfill this property, we prove in the following statement that some of them do.

\begin{proposition}
\label{prop:upperSet}
Any tubing on a geodesic between two tubings~$\tubing$ and~$\tubing'$ in the flip graph~$\flipGraph(\graphG)$ contains any common upper set to the inclusion posets of~$\tubing$ and~$\tubing'$.
\end{proposition}

In fact, we extend Theorem~\ref{theo:increasingDiameter} and Proposition~\ref{prop:upperSet} to all \defn{nestohedra}~\cite{Postnikov, FeichtnerSturmfels}, see Section~\ref{sec:diameter}. Finally, using Theorem~\ref{theo:increasingDiameter} and Proposition~\ref{prop:upperSet}, the lower bound on the diameter of the associahedron~\cite{Pournin}, the usual construction of graph associahedra~\cite{CarrDevadoss,Postnikov} and the diameter of graphical zonotopes, we obtain the following inequalities on the diameter~$\diam{\graphG}$ of~$\flipGraph(\graphG)$.

\begin{theorem}
For any connected graph~$\graphG$ with~$n+1$ vertices and~$e$ edges, the diameter~$\diam{\graphG}$ of the flip graph~$\flipGraph(\graphG)$ is bounded by
\[
\max(e,2n - 18) \le \diam{\graphG} \le \binom{n+1}{2}.
\]
\end{theorem}

In Section~\ref{sec:hamiltonicity}, we study the Hamiltonicity of~$\flipGraph(\graphG)$. Based on an inductive decomposition of graph associahedra, we show the following statement.

\begin{theorem}
\label{theo:Hamiltonian}
For any graph~$\graphG$ with at least two edges, the flip graph~$\flipGraph(\graphG)$ is Hamiltonian.
\end{theorem}


\section{Preliminaries}
\label{sec:preliminaries}


\subsection{Tubings, nested complex, and graph associahedron}

Let~$\ground$ be an $(n+1)$-elements ground set, and let~$\graphG$ be a simple graph on~$\ground$ with~$\connectedComponents{\graphG}$ connected components. We denote by~$\graphG{}[U]$ the subgraph of~$\graphG$ induced by a subset~$U$ of~$\ground$.

A \defn{tube} of~$\graphG$ is a non-empty subset~$\tube$ of~$\ground$ that induces a connected subgraph of~$\graphG$. A tube is \defn{proper} if it does not induce a connected component of~$\graphG$. The set of all tubes of~$\graphG$ is called the \defn{graphical building set} of~$\graphG$ and denoted by~$\building(\graphG)$. We moreover denote by~$\building(\graphG)_{\max}$ the set of inclusion maximal tubes of~$\building(\graphG)$, \ie the vertex sets of connected components of~$\graphG$.

Two tubes~$\tube$ and~$\tube'$ are \defn{compatible} if they are 
\begin{itemize}
\item nested, \ie $\tube \subseteq \tube'$ or~$\tube' \subseteq \tube$, or
\item disjoint and non-adjacent, \ie $\tube \cup \tube'$ is not a tube of~$\graphG$.
\end{itemize}
A \defn{tubing} on~$\graphG$ is a set of pairwise compatible tubes of~$\graphG$. A tubing is \defn{proper} if it contains only proper tubes and~\defn{loaded} if it contains~$\building(\graphG)_{\max}$. Since inclusion maximal tubes are compatible with all tubes, we can transform any tubing~$\tubing$ into a proper tubing~${\tubing \ssm \building(\graphG)_{\max}}$ or into a loaded tubing~$\tubing \cup \building(\graphG)_{\max}$, and we switch along the paper to whichever version suits better the current purpose.  Observe by the way that maximal tubings are automatically loaded. \fref{fig:tubings} illustrates these notions on a graph with $9$ vertices.

\begin{figure}[h]
  \capstart
  \centerline{\includegraphics[scale=1]{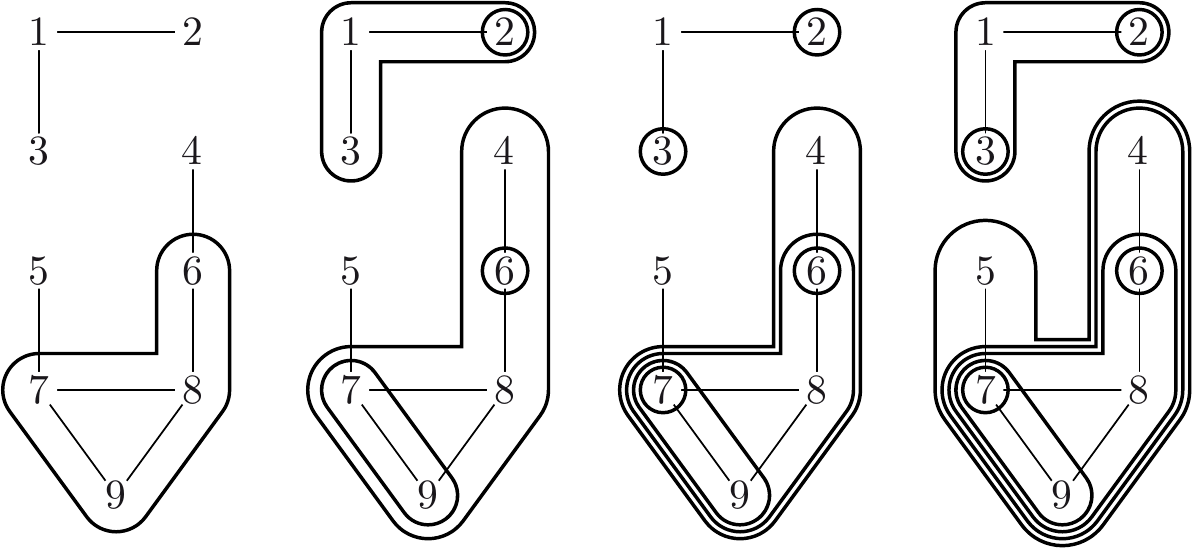}}
  \caption{A proper tube, a tubing, a maximal proper tubing, and a maximal (loaded) tubing.}
  \label{fig:tubings}
\end{figure}

The \defn{nested complex} on~$\graphG$ is the simplicial complex~$\nestedComplex(\graphG)$ of all proper tubings on~$\graphG$. This complex is known to be the boundary complex of the \defn{graph associahedron}~$\Asso(\graphG)$, which is an $(n+1-\connectedComponents{\graphG})$-dimensional simple polytope. This polytope was first constructed in~\cite{CarrDevadoss, Devadoss}\footnote{The definition used in~\cite{CarrDevadoss, Devadoss} slightly differs from ours for disconnected graphs, but our results still hold in their framework.} and later in the more general context of nestohedra in~\cite{Postnikov, FeichtnerSturmfels, Zelevinsky}. In this paper, we do not need these geometric realizations since we only consider combinatorial properties of the nested complex~$\nestedComplex(\graphG)$. In fact, we focus on the \defn{flip graph}~$\flipGraph(\graphG)$ whose vertices are maximal proper tubings on~$\graphG$ and whose edges connect adjacent maximal proper tubings, \ie which only differ by two tubes. We refer to \fref{fig:exmFlip} for an example, and to Section~\ref{subsec:flips} for a description of flips. To avoid confusion, we always use the term \emph{edge} for the edges of the graph~$\graphG$, and the term \emph{flip} for the edges of the flip graph~$\flipGraph(\graphG)$. To simplify the presentation, it is sometimes more convenient to consider the \defn{loaded flip graph}, obtained from~$\flipGraph(\graphG)$ by loading all its vertices with~$\building(\graphG)_{\max}$, and still denoted by~$\flipGraph(\graphG)$. Note that only proper tubes can be flipped in each maximal tubing on the loaded flip graph.

Observe that if~$\graphG$ is disconnected with connected components~$\graphG_i$, for ${i \in [\connectedComponents{\graphG}]}$, then the nested complex~$\nestedComplex(\graphG)$ is the join of the nested complexes~$\nestedComplex(\graphG_i)$, the graph associahedron~$\Asso(\graphG)$ is the Cartesian product of the graph associahedra~$\Asso(\graphG_i)$, and the flip graph~$\flipGraph(\graphG)$ is the Cartesian product of the flip graphs~$\flipGraph(\graphG_i)$. In many places, this allows us to restrict our arguments to connected graphs.

\begin{example}[Classical polytopes]
For certain families of graphs, the graph associahedra turn out to coincide (combinatorially) with classical polytopes (see \fref{fig:classicAssociahedra3}):
\begin{enumerate}[(i)]
\item the path associahedron~$\Asso(\pathG_{n+1})$ coincides with the $n$-dimensional associahedron,
\item the cycle associahedron~$\Asso(\cycleG_{n+1})$ coincides with the $n$-dimensional cyclohedron,
\item the complete graph associahedron~$\Asso(\completeG_{n+1})$ coincides with the $n$-dimensional permutahedron~$\Perm(n) \eqdef \conv \set{(\sigma(1), \dots, \sigma(n+1))}{\sigma \in \fS_{n+1}}$.
\end{enumerate}

\begin{figure}[hbtp]
  \capstart
  \centerline{\includegraphics[width=\textwidth]{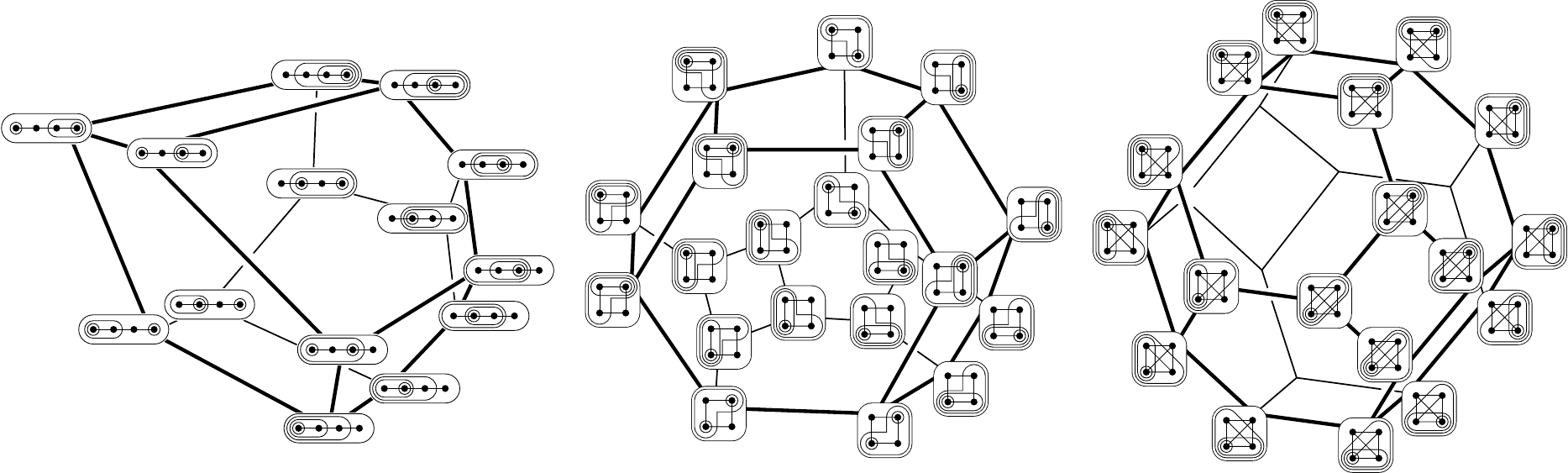}}
  \caption{The associahedron, the cyclohedron, and the permutahedron are graph associahedra.}
  \label{fig:classicAssociahedra3}
\end{figure}
\end{example}


\subsection{Spines}
\label{subsec:spines}

Spines provide convenient representations of the tubings on~$\graphG$. Given a tubing~$\tubing$ on~$\graphG$, the corresponding \defn{spine}~$\spine$ is the Hasse diagram of the inclusion poset on~$\tubing \cup \building(\graphG)_{\max}$, where the node corresponding to a tube~$\tube \in \tubing \cup \building(\graphG)_{\max}$ is labeled by~$\lab(\tube,\tubing) \eqdef \tube \ssm \bigcup \set{\tube' \in \tubing}{\tube' \subsetneq \tube}$. See~\fref{fig:exmFlip}.

The compatibility condition on the tubes of~$\tubing$ implies that the spine~$\spine$ is a rooted forest, where roots correspond to elements of~$\building(\graphG)_{\max}$. Spines are in fact called $\building(\graphG)$-forests in~\cite{Postnikov}.
The labels of~$\spine$ define a partition of the vertex set of~$\graphG$. The tubes of~${\tubing \cup \building(\graphG)_{\max}}$ are the descendants sets~$\descendants(\snode,\spine)$ of the nodes~$\snode$ of the forest~$\spine$, where~$\descendants(\snode, \spine)$ denotes the union of the labels of the descendants of~$\snode$ in~$\spine$, including~$\snode$ itself. The tubing~$\tubing \cup \building(\graphG)_{\max}$ is maximal if and only if all labels are singletons, and we then identify nodes with their labels, see~\fref{fig:exmFlip}.

Let~$\tubing, \bar\tubing$ be tubings on~$\graphG$ with corresponding spines~$\spine, \bar\spine$. Then~$\bar\tubing \subseteq \tubing$ if and only if~$\bar\spine$ is obtained from~$\spine$ by edge contractions. We say that~$\spine$ \defn{refines}~$\bar\spine$, that~$\bar\spine$ \defn{coarsens}~$\spine$, and we write~$\bar\spine \coarsens \spine$.
Given any node~$\snode$ of~$\spine$, we denote by~$\spine_{\snode}$ the \defn{subspine} of~$\spine$ induced by all descendants of~$\snode$ in~$\spine$, including~$\snode$ itself.


\subsection{Flips}
\label{subsec:flips}

As already mentioned, the nested complex~$\nestedComplex(\graphG)$ is a simplicial sphere. It follows that there is a natural flip operation on maximal proper tubings on~$\graphG$. Namely, for any maximal proper tubing~$\tubing$ on~$\graphG$ and any tube~$\tube \in \tubing$, there exists a unique proper tube~$\tube' \notin \tubing$ of~$\graphG$ such that~$\tubing' \eqdef \tubing \symdif \{\tube,\tube'\}$ is again a proper tubing on~$\graphG$ (where~$\symdif$ denotes the symmetric difference operator). We denote this flip by~$\tubing \flip \tubing'$. This flip operation can be explicitly described both in terms of tubings and spines as follows.

Consider a tube~$\tube$ in a maximal proper tubing~$\tubing$, with~$\lab(\tube, \tubing) = \{v\}$. Let~$\bar \tube$ denote the smallest element of~$\tubing \cup \building(\graphG)_{\max}$ strictly containing~$\tube$, and denote its label by~$\lab(\bar \tube, \tubing) = \{v'\}$. Then the unique tube~$\tube'$ such that~$\tubing' \eqdef \tubing \symdif \{\tube,\tube'\}$ is again a proper tubing on~$\graphG$ is the connected component of the induced subgraph~$\graphG{}[\bar \tube \ssm \{v\}]$ containing~$v'$. See \fref{fig:exmFlip}.

This description translates to spines as follows. The flip between the tubings~$\tubing$ and~$\tubing'$ corresponds to a rotation between the corresponding spines~$\spine$ and~$\spine'$. This operation is local: it only perturbs the nodes~$v$ and~$v'$ and their children. More precisely, $v$ is a child of~$v'$ in~$\spine$, and becomes the parent of~$v'$ in~$\spine'$. Moreover, the children of~$v$ in~$\spine$ contained in~$\tube'$ become children of~$v'$ in~$\spine'$. All other nodes keep their parents. See \fref{fig:exmFlip}.

\begin{figure}[hbtp]
 \centerline{
 	\begin{overpic}[scale=1]{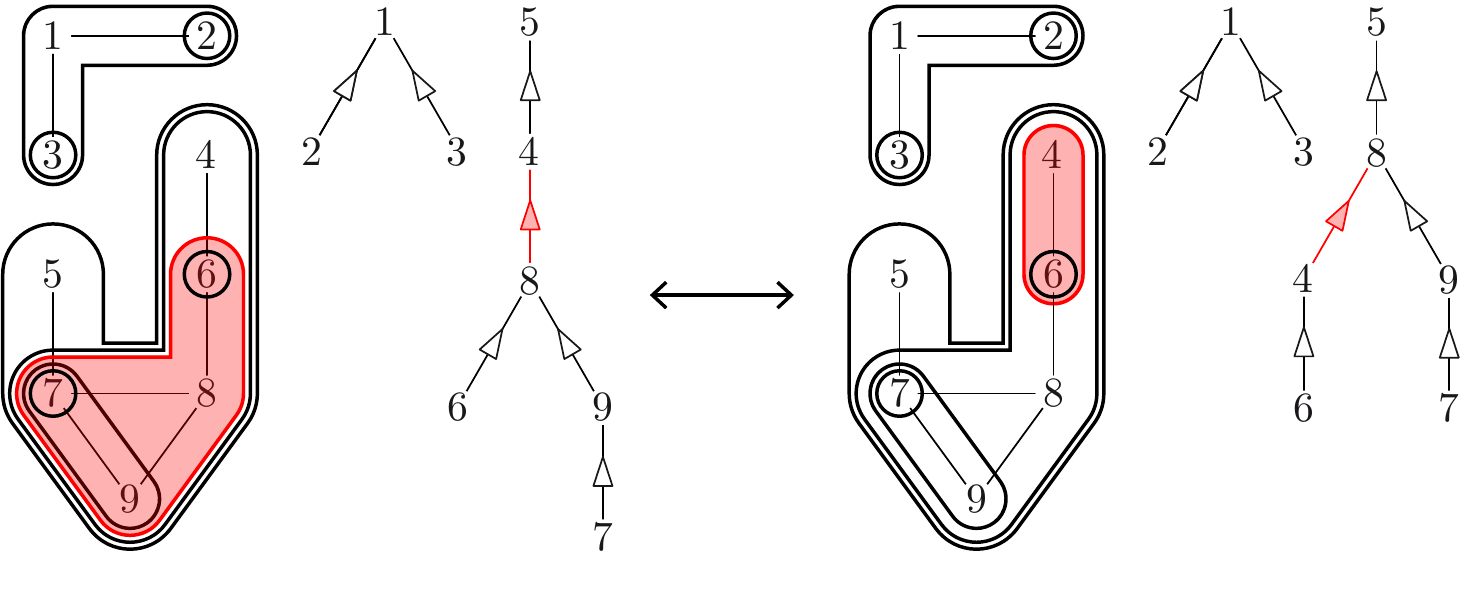}
		\put(48,22){flip}
		\put(5,1){tubing~$\tubing$}
		\put(28,1){spine~$\spine$}
		\put(63,1){tubing~$\tubing'$}
		\put(85,1){spine~$\spine'$}
	\end{overpic}
 }
 \caption{The flip of a proper tube (shaded, red) in a maximal tubing seen both on the tubings and on the corresponding spines.}
  \label{fig:exmFlip}
\end{figure}
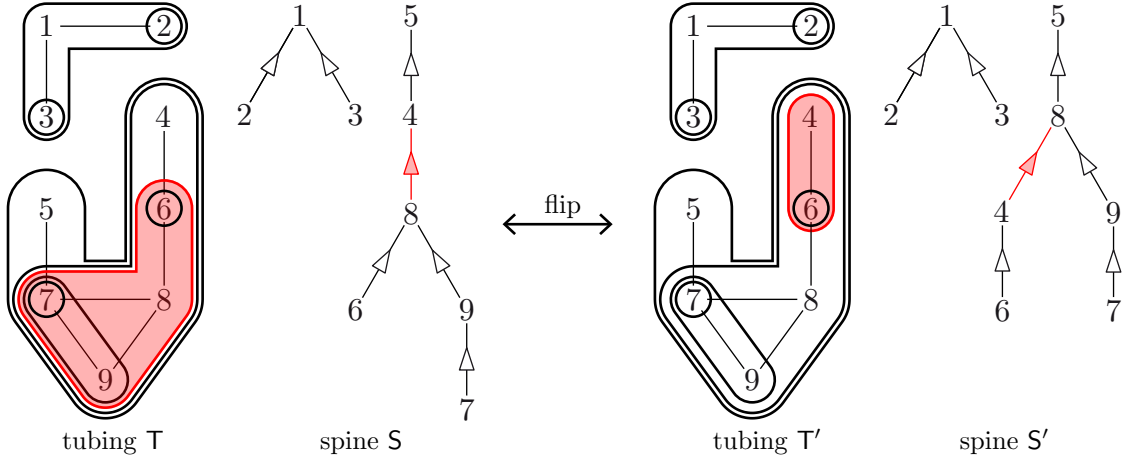


\section{Diameter}
\label{sec:diameter}

Let~$\diam{\graphG}$ denote the diameter of the flip graph~$\flipGraph(\graphG)$. For example, for the complete graph~$\completeG_{n+1}$, the diameter of the $n$-dimensional permutahedron is~${\diam{K_{n+1}} = \binom{n+1}{2}}$, while for the path~$\pathG_{n+1}$, the diameter of the classical $n$-dimensional associahedron is~${\diam{\pathG_{n+1}} = 2n - 4}$ for~$n > 9$, by results of~\cite{SleatorTarjanThurston, Pournin}. We discuss in this section properties of the diameter~$\diam{\graphG}$ and of the geodesics in the flip graph~$\flipGraph(\graphG)$. The results of Section~\ref{subsec:nonDecreasingDiam} are extended to nestohedra in Section~\ref{subsec:nestohedra}. We prefer to present the ideas first on graph associahedra as they prepare the intuition for the more technical proofs on nestohedra.


\subsection{Non-decreasing diameters}
\label{subsec:nonDecreasingDiam}

Our first goal is to show that~$\diam{\cdot}$ is non-decreasing.

\begin{theorem}
\label{theo:increasingDiameter2}
$\diam{\bar\graphG} \le \diam{\graphG}$ for any two graphs~$\graphG, \bar\graphG$ such that~$\bar\graphG \subseteq \graphG$.
\end{theorem}

\begin{remark}
We could prove this statement by a geometric argument, using the construction of the graph associahedron of M.~Carr and S.~Devadoss~\cite{CarrDevadoss}. Indeed, it follows from~\cite{CarrDevadoss} that the graph associahedron~$\Asso(\graphG)$ can be obtained from the graph associahedron~$\Asso(\bar\graphG)$ by successive face truncations. Geometrically, this operation replaces the truncated face~$\face$ by its Cartesian product with a simplex of codimension~$\dim(\face) + 1$. Therefore, a path in the graph of~$\Asso(\graphG)$ naturally projects to a shorter path in the graph of~$\Asso(\bar\graphG)$. Our proof is a purely combinatorial translation of this geometric intuition. It has the advantage not to rely on the results of~\cite{CarrDevadoss} and to help formalizing the argument.
\end{remark}

Observe first that deleting an isolated vertex in~$\graphG$ does not change the nested complex~$\nestedComplex(\graphG)$. We can thus assume that the graphs~$\graphG$ and~$\bar\graphG$ have the same vertex set and that~${\bar\graphG = \graphG \ssm \{(u,v)\}}$ is obtained by deleting a single edge~$(u,v)$ from~$\graphG$. We define below a map~$\surjection$ from tubings on~$\graphG$ to tubings on~$\bar\graphG$ which induces a surjection from the flip graph~$\flipGraph(\graphG)$ onto the flip graph~$\flipGraph(\bar\graphG)$. For consistency, we use~$\tube$ and~$\tubing$ for tubes and tubings of~$\graphG$ and~$\bar\tube$ and~$\bar\tubing$ for tubes and tubings~of~$\bar\graphG$. 

Given a tube~$\tube$ of~$\graphG$ (proper or not), define~$\surjection(\tube)$ to be the coarsest partition of~$\tube$ into tubes of~$\bar\graphG$. In other words, $\surjection(\tube) = \{\tube\}$ if~$(u,v)$ is not an isthmus of~$\graphG{}[\tube]$, and otherwise $\surjection(\tube) = \{\bar\tube_u, \bar\tube_v\}$ where~$\bar\tube_u$ and~$\bar\tube_v$ are the vertex sets of the connected components of~$\bar\graphG{}[\tube]$ containing~$u$ and~$v$ respectively. For a set of tubes~$\tubing$ of~$\graphG$, define~$\surjection(\tubing) \eqdef \bigcup_{\tube \in \tubing} \surjection(\tube)$. See \fref{fig:exmSurjection} for an illustration.

\begin{figure}[hbtp]
 \centerline{\includegraphics[scale=1]{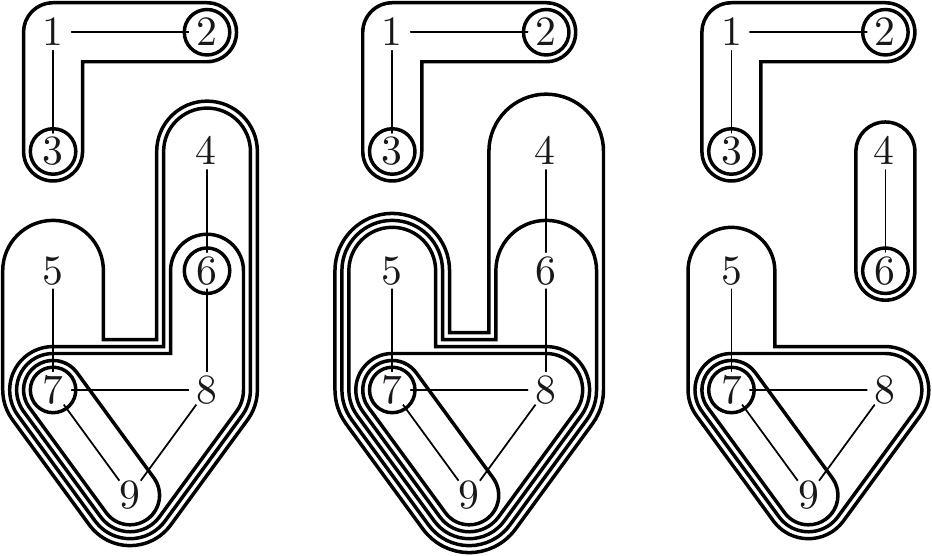}}
 \caption{Two maximal tubings (left and middle) with the same image by the map~$\surjection$ (right). The middle tubing is the preimage of the rightmost tubing obtained by the process decribed in the proof of Corollary~\ref{coro:surjectionFlips} with~$u = 6$ and~$v = 8$.}
  \label{fig:exmSurjection}
\end{figure}

\begin{lemma}
\label{lem:surjectionIncreasing}
For any tubing~$\tubing$ on~$\graphG$, the set~$\surjection(\tubing)$ is a tubing on~$\bar\graphG$ and~${|\tubing| \le |\surjection(\tubing)|}$.
\end{lemma}

\begin{proof}
It is immediate to see that~$\surjection$ sends tubings on~$\graphG$ to tubings on~$\bar\graphG$.
We prove by induction on~$|\tubing|$ that~${|\tubing| \le |\surjection(\tubing)|}$. Consider a non-empty tubing~$\tubing$, and let~$\tube$ be an inclusion maximal tube of~$\tubing$. By induction hypothesis, $|\tubing \ssm \{\tube\}| \le |\surjection(\tubing \ssm \{\tube\})|$. We now distinguish two cases:
\begin{enumerate}[(i)]
\item If~$(u,v)$ is an isthmus of~$\graphG{}[\tube]$, then~$\surjection(\tube) = \{\bar\tube_u, \bar\tube_v\} \not\subseteq \surjection(\tubing \ssm \{\tube\})$. Indeed, since~$\bar\tube_u$ and~$\bar\tube_v$ are adjacent in~$\graphG$, two tubes of~$\tubing$ whose images by~$\surjection$ produce~$\bar\tube_u$ and~$\bar\tube_v$ must be nested.~Therefore, one of them contains both~$\bar\tube_u$ and~$\bar\tube_v$, and thus equals~$\tube = \bar\tube_u \cup \bar\tube_v$ by maximality of~$\tube$ in~$\tubing$.
\item If~$(u,v)$ is not an isthmus of~$\graphG{}[\tube]$, then~$\surjection(\tube) = \{\tube\} \not\subseteq \surjection(\tubing \ssm \{\tube\})$. Indeed, if~$\tube' \in \tubing$ is such that~$\tube \in \surjection(\tube')$, then~$\tube \subseteq \tube'$ and thus~$\tube = \tube'$ by maximality of~$\tube$ in~$\tubing$.
\end{enumerate}
We conclude that~$|\surjection(\tubing)| \ge |\surjection(\tubing \ssm \{\tube\})| + 1 \ge |\tubing \ssm \{\tube\}| + 1 = |\tubing|$.
\end{proof}

\begin{corollary}
\label{coro:surjectionFlips}
The map~$\surjection$ induces a graph surjection from the loaded flip graph~$\flipGraph(\graphG)$ onto the loaded flip graph~$\flipGraph(\bar\graphG)$, \ie a surjective map from maximal tubings on~$\graphG$ to maximal tubings on~$\bar\graphG$ such that adjacent tubings on~$\graphG$ are sent to identical or adjacent tubings on~$\bar\graphG$.
\end{corollary}

\begin{proof}
Let~$\bar\tubing$ be a tubing on~$\bar\graphG$. If all tubes of~$\bar\tubing$ containing $u$ also contain~$v$ (or the opposite), then~$\bar\tubing$ is a tubing on~$\graphG$ and~$\surjection(\bar\tubing) = \bar\tubing$. Otherwise, let~$\bar\tubing_u$ denote the set of tubes of~$\bar\tubing$ containing~$u$ but not~$v$ and~$\bar\tube_v$ denote the maximal tube containing~$v$ but not~$u$. Then~$(\bar\tubing \ssm \bar\tubing_u) \cup \set{\bar\tube_u \cup \bar\tube_v}{\bar\tube_u \in \bar\tubing_u}$ is a tubing on~$\graphG$ whose image by~$\surjection$ is~$\bar\tubing$. See \fref{fig:exmSurjection} for an illustration. The map~$\surjection$ is thus surjective from tubings on~$\graphG$ to tubings on~$\bar\graphG$. Moreover, any preimage~$\tubing_\circ$ of a maximal tubing~$\bar\tubing$ can be completed into a maximal tubing~$\tubing$ with~$\surjection(\tubing) \supseteq \surjection(\tubing_\circ) = \bar\tubing$, and thus satisfying~$\surjection(\tubing) = \bar\tubing$ by maximality~of~$\bar\tubing$.

Remember that two distinct maximal tubings on~$\graphG$ are adjacent if and only if they share precisely~$|\ground| - 1$ common tubes. Consider two adjacent maximal tubings~$\tubing, \tubing'$ on~$\graphG$, so that ${|\tubing \cap \tubing'| = |\ground| - 1}$. Since~$\surjection(\tubing \cap \tubing') \subseteq \surjection(\tubing) \cap \surjection(\tubing')$ and $|\surjection(\tubing \cap \tubing')| \ge |\tubing \cap \tubing'|$ by Lemma~\ref{lem:surjectionIncreasing}, we have~$|\surjection(\tubing) \cap \surjection(\tubing')| \ge |\tubing \cap \tubing'| = |\ground| - 1$. Therefore, the tubings~$\surjection(\tubing), \surjection(\tubing')$ are adjacent if~$|\surjection(\tubing) \cap \surjection(\tubing')| = |\tubing \cap \tubing'|$ and identical if~${|\surjection(\tubing) \cap \surjection(\tubing')| > |\tubing \cap \tubing'|}$.
\end{proof}

\begin{remark}
\label{rem:shuffleChains}
We can in fact precisely describe the preimage~$\surjection^{-1}(\bar\tubing)$ of a maximal tubing~$\bar\tubing$ on~$\bar\graphG$ as follows. As in the previous proof, let~$\bar\tubing_u$ denote the chain of tubes of~$\bar\tubing$ containing~$u$ but not~$v$ and similarly~$\bar\tubing_v$ denote the chain of tubes of~$\bar\tubing$ containing~$v$ but not~$u$. Any linear extension~$\linearExtensionTubing$ of these two chains defines a preimage of~$\bar\tubing$ where the tubes of~$\bar\tubing_u \cup \bar\tubing_v$ are replaced by the tubes~$\bigcup \set{\tube' \in \linearExtensionTubing}{\tube' \le_{\linearExtensionTubing} \tube}$ for~$\tube \in \linearExtensionTubing$. In terms of spines, this translates to shuffling the two chains corresponding to~$\bar\tubing_u$ and~$\bar\tubing_v$. Details are left to the reader.
\end{remark}

\begin{proof}[Proof of Theorem~\ref{theo:increasingDiameter2}]
Consider two maximal tubings~$\bar\tubing, \bar\tubing'$ on~$\bar\graphG$. Let~$\tubing, \tubing'$ be maximal~loaded tubings on~$\graphG$ such that~$\surjection(\tubing) = \bar\tubing$ and~$\surjection(\tubing') = \bar\tubing'$ (surjectivity of~$\surjection$), and ${\tubing = \tubing_0, \dots, \tubing_\ell = \tubing'}$ be a geodesic between them ($\ell \le \diam{\graphG}$). Deleting repetitions in the sequence~${\bar\tubing = \surjection(\tubing_0), \dots, \surjection(\tubing_\ell) = \bar\tubing'}$ yields a path from~$\bar\tubing$ to~$\bar\tubing'$ (Corollary~\ref{coro:surjectionFlips}) of length at most~$\ell \le \diam{\graphG}$. So~${\diam{\graphG} \ge \diam{\bar\graphG}}$.
\end{proof}


\subsection{Extension to nestohedra}
\label{subsec:nestohedra}

The results of the previous section can be extended to the nested complex on an arbitrary building set. Although the proofs are more abstract and technical, the ideas behind are essentially the same. We recall the definitions of building set and nested complex needed here and refer to~\cite{CarrDevadoss, Postnikov, FeichtnerSturmfels, Zelevinsky} for more details and motivation.

A \defn{building set} on a ground set~$\ground$ is a collection~$\building$ of non-empty subsets of~$\ground$~such~that
\begin{enumerate}[(B1)]
\item if~$\block, \block' \in \building$ and~$\block \cap \block' \ne \varnothing$, then~$\block \cup \block' \in \building$, and
\item $\building$ contains all singletons~$\{v\}$ for~$v \in \ground$.
\end{enumerate} 
We denote by~$\building_{\max}$ the set of inclusion maximal elements of~$\building$ and call \defn{proper} the elements of~$\building \ssm \building_{\max}$. The building set is \defn{connected} if~$\building_{\max} = \{\ground\}$.  Graphical building sets are particular examples, and connected graphical building sets correspond to connected graphs.

A \defn{$\building$-nested set} on~$\building$ is a subset~$\nested$ of~$\building$ such that
\begin{enumerate}[(N1)]
\item for any~$\nestelem, \nestelem' \in \nested$, either~$\nestelem \subseteq \nestelem'$ or~$\nestelem' \subseteq \nestelem$ or~$\nestelem \cap \nestelem' = \varnothing$, and
\item for any~$k \ge 2$ pairwise disjoint sets~$\nestelem_1, \dots, \nestelem_k \in \nested$, the union~$\nestelem_1 \cup \dots \cup \nestelem_k$ is not in~$\building$.
\end{enumerate}
As before, a $\building$-nested set~$\nested$ is \defn{proper} if~$\nested \cap \building_{\max} = \varnothing$ and \defn{loaded} if~$\building_{\max} \subseteq \nested$. The \defn{$\building$-nested complex} is the $(|\ground| - |\building_{\max}|)$-dimensional simplicial complex~$\nestedComplex(\building)$ of all proper nested sets on~$\building$. As in the graphical case, the $\building$-nested complex can be realized geometrically as the boundary complex of the polar of the \defn{nestohedron}~$\Nest(\building)$, constructed \eg in~\cite{Postnikov, FeichtnerSturmfels, Zelevinsky}. We denote by~$\diam{\building}$ the diameter of the graph~$\flipGraph(\building)$ of~$\Nest(\building)$. As in the previous section, it is more convenient to regard the vertices of~$\flipGraph(\building)$ as maximal loaded nested sets.

The \defn{spine} of a nested set~$\nested$ is the Hasse diagram of the inclusion poset of~$\nested \cup \building_{\max}$. Spines are called $\building$-forests in~\cite{Postnikov}. The definitions and properties of Section~\ref{subsec:spines} extend to general building sets, see~\cite{Postnikov} for details.

We shall now prove the following generalization of Theorem~\ref{theo:increasingDiameter2}.

\begin{theorem}
\label{theo:increasingDiam3}
$\diam{\bar\building} \le \diam{\building}$ for any two building sets~$\building, \bar\building$ on~$\ground$ such that~$\bar\building \subseteq \building$.
\end{theorem}

The proof follows the same line as that of Theorem~\ref{theo:increasingDiameter2}. We first define a map~$\surjection$ which transforms elements of~$\building$ to subsets of~$\bar\building$ as follows: for $\block \in \building$ (proper or not), define~$\surjection(\block)$ as the coarsest partition of~$\block$ into elements of~$\bar\building$. Observe that~$\surjection(\block)$ is well-defined since~$\bar\building$ is a building set, and that the elements of~$\surjection(\block)$ are precisely the inclusion maximal elements of~$\bar\building$ contained in~$\block$. For a nested set~$\nested$ on~$\building$, we define~$\surjection(\nested) \eqdef \bigcup_{\nestelem \in \nested} \surjection(\nestelem)$. The following statement is similar to Lemma~\ref{lem:surjectionIncreasing}.

\begin{lemma}
For any nested set~$\nested$ on~$\building$, the image~$\surjection(\nested)$ is a nested set on~$\bar\building$ and~$|\nested| \le |\surjection(\nested)|$.
\end{lemma}

\begin{proof}
Consider a nested set~$\nested$ on~$\building$. To prove that~$\surjection(\nested)$ is a nested set on~$\bar\building$, we start with condition~(N1). Let~$\bar\nestelem, \bar\nestelem' \in \surjection(\nested)$ and let~$\nestelem, \nestelem' \in \nested$ such that~$\bar\nestelem \in \surjection(\nestelem)$ and~$\bar\nestelem' \in \surjection(\nestelem')$. Since~$\nested$ is nested, we can distinguish two cases:
\begin{itemize}
\item Assume that~$\nestelem$ and~$\nestelem'$ are disjoint. Then~$\bar\nestelem \cap \bar\nestelem' = \varnothing$ since~$\bar\nestelem \subseteq \nestelem$ and~$\bar\nestelem' \subseteq \nestelem'$.
\item Assume that~$\nestelem$ and~$\nestelem'$ are nested, \eg $\nestelem \subseteq \nestelem'$. If~$\bar\nestelem \cap \bar\nestelem' \ne \varnothing$, then~$\bar\nestelem \cup \bar\nestelem'$ is in~$\bar\building$ and is a subset of~$\nestelem'$. By maximality of~$\bar\nestelem'$ in~$\nestelem'$, we obtain~$\bar\nestelem \cup \bar\nestelem' = \bar\nestelem'$, and thus~$\bar\nestelem \subseteq \bar\nestelem'$.
\end{itemize}
To prove Condition~(N2), consider pairwise disjoint elements~$\bar\nestelem_1, \dots, \bar\nestelem_k \in \surjection(\nested)$ and~${\nestelem_1, \dots, \nestelem_k \in \nested}$ such that~$\bar\nestelem_i \in \surjection(\nestelem_i)$. We assume by contradiction that~$\bar\nestelem \eqdef \bar\nestelem_1 \cup \dots \cup \bar\nestelem_k \in \bar\building$ and we prove that~$\nestelem \eqdef \nestelem_1 \cup \dots \cup \nestelem_k \in \building$. Indeed, $\bar\nestelem, \nestelem_1, \dots, \nestelem_k$ all belong to~$\building$ and~$\bar\nestelem \cap \nestelem_i \ne \varnothing$ (it contains~$\bar\nestelem_i$) so that~$\bar\nestelem \cup \nestelem$ also belongs to~$\building$ by multiple applications of Property~(B1) of building sets. Moreover,~$\bar\nestelem \subseteq \nestelem$ so that~$\nestelem = \bar\nestelem \cup \nestelem \in \building$. Finally, we conclude distinguishing two cases:
\begin{itemize}
\item If there is~$i \in [k]$ such that~$\nestelem_i$ contains all~$\nestelem_j$, then~$\nestelem_i$ contains all~$\bar\nestelem_j$ and thus~$\bar\nestelem$. This contradicts the maximality of~$\bar\nestelem_i$ in~$\nestelem_i$ since~$\bar\nestelem_i \subsetneq \bar\nestelem \in \bar\building$.
\item Otherwise, merging intersecting elements allows us to assume that~$\nestelem_1, \dots, \nestelem_k$ are pairwise disjoint and~$\nestelem \eqdef \nestelem_1 \cup \dots \cup \nestelem_k \in \building$ contradicts Condition~(N2) for~$\nested$. 
\end{itemize}
This concludes the proof that~$\surjection(\nested)$ is a nested set on~$\bar\building$.

We now prove that~$|\nested| \le |\surjection(\nested)|$ by induction on~$|\nested|$. Consider a non-empty nested set~$\nested$ and let~$\nestelem_\circ$ be an inclusion maximal element of~$\nested$. By induction hypothesis, $|\nested \ssm \{\nestelem_\circ\}| \le |\surjection(\nested \ssm \{\nestelem_\circ\})|$. Let~$\surjection(\nestelem_\circ) = \{\bar\nestelem_1, \dots, \bar\nestelem_k\}$. Consider~$\nestelem_1, \dots, \nestelem_k \in \nested$ such that~$\bar\nestelem_i \in \surjection(\nestelem_i)$, and let~${\nestelem \eqdef \nestelem_1 \cup \dots \cup \nestelem_k}$. Since~$\nestelem_\circ, \nestelem_1, \dots, \nestelem_k$ all belong to~$\building$ and~$\nestelem_\circ \cap \nestelem_i \ne \varnothing$ (it contains~$\bar\nestelem_i$), we have~$\nestelem_\circ \cup \nestelem \in \building$ by multiple applications of Property~(B1) of building sets. Moreover,~$\nestelem_\circ \subseteq \nestelem$ so that~$\nestelem = \nestelem_\circ \cup \nestelem \in \building$. It follows by Condition~(N2) on~$\nested$ that there is~$i \in [k]$ such that~$\nestelem_i$ contains all~$\nestelem_j$, and thus~$\nestelem_\circ \subseteq \nestelem_i$. We obtain that~$\nestelem_\circ = \nestelem_i$ by maximality of~$\nestelem_\circ$. We conclude that~$\nestelem_\circ$ is the only element of~$\nested$ such that~$\bar\nestelem_i \in \surjection(\nestelem_\circ)$, so that~$|\surjection(\nested)| \ge |\surjection(\nested \ssm \{\nestelem_\circ\})| + 1 \ge |\nested \ssm \{\nestelem_\circ\}| + 1 = |\nested|$.
\end{proof}

\begin{corollary}
\label{coro:surjectionNested}
The map~$\surjection$ induces a graph surjection from the loaded flip graph~$\flipGraph(\building)$ onto the loaded flip graph~$\flipGraph(\bar\building)$, \ie a surjective map from maximal nested sets on~$\building$ to maximal nested sets on~$\bar\building$ such that adjacent nested sets on~$\building$ are sent to identical or adjacent nested sets on~$\bar\building$.
\end{corollary}

\begin{proof}
To prove the surjectivity, consider a nested set~$\bar\nested$ on~$\bar\building$. The elements of~$\bar\nested$ all belong to~$\building$ and satisfy Condition~(N1) for nested sets. It remains to transform the elements in~$\bar\nested$ which violate Condition~(N2). If there is no such violation, then~$\bar\nested$ is a nested set on~$\building$ and~$\surjection(\bar\nested) = \bar\nested$. Otherwise, consider pairwise disjoint elements~$\bar\nestelem_1, \dots, \bar\nestelem_k$ of~$\bar\nested$ such that~$\nestelem \eqdef \bar\nestelem_1 \cup \dots \cup \bar\nestelem_k$ is in~$\building$ and is maximal for this property. Consider the subset~$\bar\nested' \eqdef \big( \bar\nested \ssm \{\bar\nestelem_1\} \big) \cup \{\nestelem\}$ of~$\building$. Observe that:
\begin{itemize}
\item $\bar\nested'$ still satisfies  Condition~(N1). Indeed, if~$\bar\nestelem \in \bar\nested$ is such that~$\nestelem \cap \bar\nestelem \ne \varnothing$, then~$\bar\nestelem$ intersects at least one element~$\bar\nestelem_i$. Since~$\bar\nested$ is nested, $\bar\nestelem \subseteq \bar\nestelem_i$ or~$\bar\nestelem_i \subseteq \bar\nestelem$. In the former case, $\bar\nestelem \subseteq \nestelem$ and we are done. In the latter case, $\bar\nestelem$ and the elements~$\bar\nestelem_j$ disjoint from~$\bar\nestelem$ would contradict the maximality of~$\nestelem$.
\item $\bar\nested'$ still satisfies $\surjection(\bar\nested') = \bar\nested$. Indeed, $\bar\nestelem_1 \in \surjection(\nestelem)$ since~$\surjection(\nestelem) = \{\bar\nestelem_1, \dots, \bar\nestelem_k\}$. For the latter equality, observe that $\{\bar\nestelem_1, \dots, \bar\nestelem_k\}$ is a partition of~$\nestelem$ into elements of~$\bar\building$ and that a coarser partition would contradict Condition~(N2) on~$\bar\nested$.
\item $\nestelem$ cannot  be partitioned into two or more elements of~$\bar\nested'$. Such a partition would refine the partition~$\surjection(\nestelem)$, and would thus contradict again Condition~(N2) on~$\bar\nested$. Therefore, $\bar\nested'$ has strictly less violations of Condition~(N2) than~$\bar\nested$.
\item All violations of Condition~(N2) in~$\bar\nested'$ only involve elements of~$\bar\building$. Indeed, pairwise disjoint elements~$\bar\nestelem'_1, \dots, \bar\nestelem'_\ell \in \bar\nested'$ disjoint from~$\nestelem$ and such that~$\nestelem \cup \bar\nestelem'_1 \cup \dots \cup \bar\nestelem'_\ell \in \building$ would contradict the maximality of~$\nestelem$.
\end{itemize}
These four points enable us to decrease the number of violations of Condition~(N2) until we reach a nested set~$\nested$ on~$\building$ which still satisfies~$\surjection(\nested) = \bar\nested$.

The second part of the proof is identical to that of Corollary~\ref{coro:surjectionFlips}.
\end{proof}

From Corollary~\ref{coro:surjectionNested}, the proof of Theorem~\ref{theo:increasingDiam3} is identical to that of Theorem~\ref{theo:increasingDiameter2}.


\subsection{Geodesic properties}

\newcommand{\EFP}{\textsc{efp}}
\newcommand{\SNLFP}{\textsc{snlfp}}
\newcommand{\NLFP}{\textsc{nlfp}}

In this section, we focus on properties of the geodesics in the graphs of nestohedra. We consider three properties for a face~$\face$ of a polytope~$\polytope$:
\begin{description}
\item[NLFP] $\face$ has the \defn{non-leaving-face property} in~$\polytope$ if~$\face$ contains all geodesics connecting two vertices of~$\face$ in the graph of~$\polytope$.
\item[SNLFP] $\face$ has the \defn{strong non-leaving-face property} in~$\polytope$ if any path connecting two vertices~$v,w$ of~$\face$ in the graph of~$\polytope$ and leaving the face~$\face$ has at least two more steps than a geodesic between~$v$ and~$w$.
\item[EFP] $\face$ has the \defn{entering-face property} in~$\polytope$ if for any vertices~$u,v,w$ of~$\polytope$ such that~$u \notin \face$, $v,w \in \face$, and~$u$ and~$v$ are neighbors in the graph of~$\polytope$, there exists a geodesic connecting~$u$ and~$w$ whose first edge is the edge from~$u$ to~$v$.
\end{description}

For a face~$\face$ of a polytope~$\polytope$, we have~$\EFP \iff \SNLFP \implies \NLFP$. However, the reverse of the last implication is wrong: all faces of a simplex have the \NLFP{} (all vertices are at distance~$1$), but not the \SNLFP. Alternative counter-examples with no simplicial face already exist in dimension~$3$. Among classical polytopes the $n$-dimensional cube, permutahedron, associahedron, and cyclohedron all satisfy the \EFP. The \NLFP{} is further discussed in~\cite{CeballosPilaud-diameter}.

\begin{figure}
  \capstart
  \centerline{\includegraphics[width=\textwidth]{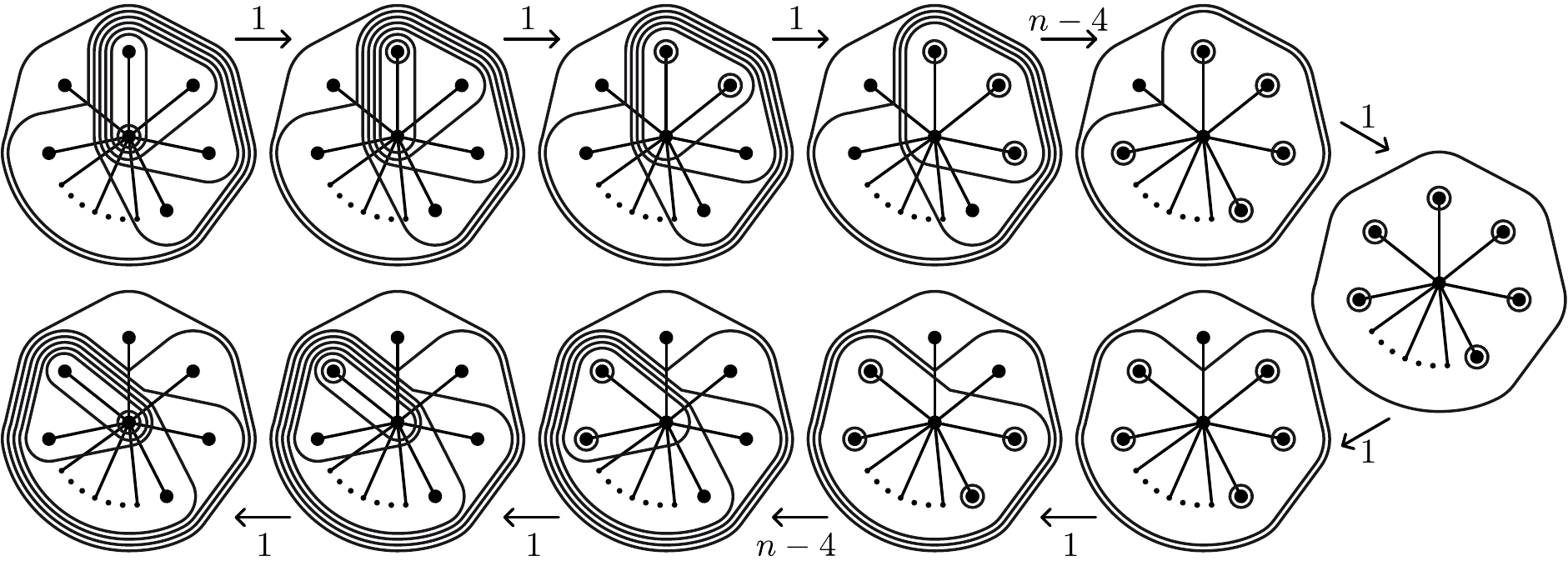}}
  \caption{A geodesic (of length~$2n$) between two maximal tubings of the star that flips their common tube (the central vertex).}
  \label{fig:ctrexmNLFP}
\end{figure}

Contrarily to the classical associahedron, not all faces of a graph associahedron have the \NLFP. A counter-example is given by the star with $n$ branches: \fref{fig:ctrexmNLFP} shows a path of length~$2n$ between two maximal tubings~$\tubing, \tubing'$, while the minimal face containing~$\tubing$ and~$\tubing'$ is an $(n-1)$-dimensional permutahedron (see the face description in~\cite[Theorem~2.9]{CarrDevadoss}) and the graph distance between~$\tubing$ and~$\tubing'$ in this face is~$\binom{n}{2}$. It turns out however that the following faces of the graph associahedra, and more generally of nestohedra, always have the \SNLFP.

\begin{lemma}
\label{lem:upperIdealFace}
We call \defn{upper ideal face} of the nestohedron~$\Nest(\building)$ a face corresponding to a loaded nested set~$\nested^\uparrow$ that satisfies the following equivalent properties:
\begin{enumerate}[(i)]
\item any element of~$\building$ not in~$\nested^\uparrow$ but compatible with~$\nested^\uparrow$ is contained in an inclusion minimal element of~$\nested^\uparrow$,
\item the set~$\lab(\nestelem, \nested^\uparrow) \eqdef \nestelem \ssm \bigcup \set{\nestelem' \in \nested^\uparrow}{\nestelem' \subsetneq \nestelem}$ is a singleton for any inclusion non-minimal element~$\nestelem$ of~$\nested^\uparrow$,
\item the forest obtained by deleting all leaves of the spine~$\spine^\uparrow$ of~$\nested^\uparrow$ forms an upper ideal of any spine refining~$\spine^\uparrow$.
\end{enumerate}
\end{lemma}

\begin{proof}
We first prove that~(i)$\implies$(ii). Assume that~$\nestelem \in \nested^\uparrow$ is not inclusion minimal and that~$\lab(\nestelem,\nested^\uparrow)$ contains two distinct elements~$v,w \in \ground$. One can then check that the maximal element of~$\building$ contained in~$\nestelem$ and containing~$v$ but not~$w$ is compatible with~$\nested^\uparrow$, but not contained in an inclusion minimal element of~$\nested^\uparrow$. This proves that~(i)$\implies$(ii).

Conversely, assume~(ii) and consider~$\block \in \building$ not in~$\nested^\uparrow$ but compatible with~$\nested^\uparrow$. Since~$\nested^\uparrow$ is loaded, there exists~$\nestelem \in \nested^\uparrow$ strictly containing~$\block$ and minimal for this property. Since~$\block$ is compatible with~$\nested^\uparrow$, we obtain that~$\lab(\nestelem, \nested^\uparrow)$ contains at least one element from~$\block$ and one from~$\nestelem \ssm \block$, and is thus not a singleton. It follows by~(ii) that~$\nestelem$ is an inclusion minimal element of~$\nested^\uparrow$, and it contains~$\block$.

The equivalence (ii)$\iff$(iii) follows directly from the definition of the spines and their labelings, and the fact that a non-singleton node in a spine can be split in a refining spine.
\end{proof}

\begin{proposition}
\label{prop:SNLFP}
Any upper ideal face of the nestohedron~$\Nest(\building)$ satisfies \SNLFP.
\end{proposition}

\begin{proof}
Consider an upper ideal face~$\face$ of~$\Nest(\building)$ corresponding to the loaded nested set~$\nested^\uparrow$. We consider the building set~$\bar\building \subseteq \building$ on~$\ground$ consisting of all elements of~$\building$ (weakly) contained in an inclusion minimal element of~$\nested^\uparrow$ together with all singletons~$\{v\}$ for elements~$v \in \ground$ not contained in any inclusion minimal element of~$\nested^\uparrow$. The reader is invited to check that~$\bar\building$ is indeed a building set on~$\ground$. It follows from Lemma~\ref{lem:upperIdealFace} that
\begin{itemize}
\item $\lab(\nestelem, \nested^\uparrow) = \nestelem$ if~$\nestelem$ is an inclusion minimal element of~$\nested^\uparrow$, 
\item $\lab(\nestelem,\nested^\uparrow) = \{v\}$ for some~$v$ not contained in any inclusion minimal element of~$\nested^\uparrow$~otherwise,
\end{itemize}
and thus that the map~$\lab(\cdot, \nested^\uparrow)$ is a bijection from~$\nested^\uparrow$ to~$\bar\building_{\max}$.

Consider the surjection~$\surjection$ from the maximal nested sets on~$\building$ to the maximal nested sets on~$\bar\building$ as defined in the previous section: $\surjection(\nested) = \bigcup_{\nestelem \in \nested} \surjection(\nestelem)$ where~$\surjection(\nestelem)$ is the coarsest partition of~$\nestelem$ into elements of~$\bar\building$. Following~\cite{SleatorTarjanThurston, CeballosPilaud-diameter}, we consider the \defn{normalization}~$\normalization$ on maximal nested sets on~$\building$ defined by~$\normalization(\nested) \eqdef \big( \surjection(\nested) \ssm \bar\building_{\max} \big) \cup \nested^\uparrow$. We claim that~$\normalization(\nested)$ is a maximal nested set~on~$\building$:
\begin{itemize}
\item it is nested since both~$\surjection(\nested) \ssm \bar\building_{\max}$ and~$\nested^\uparrow$ are themselves nested, and all elements of~${\surjection(\nested) \ssm \bar\building_{\max}}$ are contained in a minimal element of~$\nested^\uparrow$.
\item it is maximal since~$\surjection(\nested)$ is maximal by Corollary~\ref{coro:surjectionNested} and~$|\normalization(\nested)| = |\surjection(\nested)|$ because~$\lab(\cdot, \nested^\uparrow)$ is a bijection from~$\nested^\uparrow$ to~$\bar\building_{\max}$, and~$\bar\building_{\max} \subseteq \surjection(\nested)$ while~$\big( \surjection(\nested) \ssm \bar\building_{\max} \big) \cap \nested^\uparrow = \varnothing$.
\end{itemize}
It follows that the map~$\normalization$ combinatorially projects the nestohedron~$\Nest(\building)$ onto its face~$\face$.

Let~$\nested_0, \dots, \nested_\ell$ be a path in the loaded flip graph~$\flipGraph(\building)$ whose endpoints~$\nested_0, \nested_\ell$ lie in the face~$\face$, but which leaves the face~$\face$. In other words,~$\nested^\uparrow \subseteq \nested_0, \nested_\ell$ and there are~$0 \le i < j \le \ell$ such that~$\nested^\uparrow \subseteq \nested_i, \nested_j$ while~${\nested^\uparrow \not\subseteq \nested_{i+1}, \nested_{j-1}}$. We claim that
\[
\normalization(\nested_0) = \nested_0, \quad \normalization(\nested_\ell) = \nested_\ell, \quad \normalization(\nested_i) = \nested_i = \normalization(\nested_{i+1}) \quad\text{and}\quad \normalization(\nested_{j-1}) = \nested_j = \normalization(\nested_j),
\]
so that the path~$\nested_0 = \normalization(\nested_0), \dots, \normalization(\nested_\ell) = \nested_\ell$ from~$\nested_0$ to~$\nested_\ell$ in~$\face$ has length at most~$\ell-2$ after deletion of repetitions.

To prove our claim, consider a loaded nested set~$\nested$ on~$\building$ containing a maximal proper nested set~$\bar\nested$ on~$\bar\building$. Then~$\surjection(\nested) \supseteq \surjection(\bar\nested) = \bar\nested$ so that~$\surjection(\nested) = \bar\nested \cup \bar\building_{\max}$ by maximality of~$\bar\nested$. This shows~${\normalization(\nested) = \bar\nested \cup \nested^\uparrow}$. In particular, if~$\nested = \bar\nested \cup \nested^\uparrow$, then~$\normalization(\nested) = \nested$. Moreover, if~$\nested'$ is adjacent to~$\nested = \bar\nested \cup \nested^\uparrow$ and does not contain~$\nested^\uparrow$, then~$\nested'$ contains~$\bar\nested$ and~$\normalization(\nested') = \nested$. This shows the claim and concludes~the~proof.
\end{proof}

Proposition~\ref{prop:SNLFP} specializes in particular to the non-leaving-face and entering face properties for the upper set faces of graph associahedra.

\begin{proposition}
\label{prop:upperSet2}
\begin{enumerate}[(i)]
\item If~$\tubing$ and~$\tubing'$ are two maximal tubings on~$\graphG$, then any maximal tubing on a geodesic between~$\tubing$ and~$\tubing'$ in the flip graph~$\flipGraph(\graphG)$ contains any common upper set to the inclusion posets of~$\tubing$ and~$\tubing'$.
\item If~$\tubing$, $\tubing'$ and~$\tubing''$ are three maximal tubings on~$\graphG$ such that~$\tubing \ssm \{\tube\} = \tubing' \ssm \{\tube'\}$ and~$\tube'$ belongs to the maximal common upper set to the inclusion poset of~$\tubing'$ and~$\tubing''$, then there is a geodesic between~$\tubing$ and~$\tubing''$ starting by the flip from~$\tubing$ to~$\tubing'$.
\end{enumerate}
\end{proposition}

\begin{proof}
Using Proposition~\ref{prop:SNLFP}, it is enough to show that the maximal common upper set~$\tubing^\uparrow$ to the inclusion posets of~$\tubing$ and~$\tubing'$ defines an upper ideal face of~$\Asso(\graphG)$. For this, we use the characterization~(ii) of Lemma~\ref{lem:upperIdealFace}. Consider an inclusion non-minimal tube~$\tube$ of~$\tubing^\uparrow$. Let~$\tube'$ be a maximal tube of~$\tubing^\uparrow$ such that~$\tube' \subsetneq \tube$. Then~$\tube'$ has a unique neighbor~$v$ in~$\graphG{}[\tube]$ and all connected components of~$\graphG{}[\tube \ssm \{v\}]$ are both in~$\tubing$ and~$\tubing'$, thus in~$\tubing^\uparrow$. Thus~$\lab(\tube, \tubing^\uparrow) = \{v\}$.
\end{proof}

\begin{remark}
For an arbitrary building set~$\building$, the maximal common upper set~$\nested^\uparrow$ to the inclusion poset of two maximal nested sets~$\nested, \nested'$ is not always an upper ideal face of~$\Nest(\building)$. A minimal example is the building set~$\building = \big\{ \{1\}, \{2\}, \{3\}, \{1,2,3\} \big\}$ and the nested sets~$\nested = \big\{ \{1\}, \{2\}, \{1,2,3\} \big\}$ and~$\nested' = \big\{ \{2\}, \{3\}, \{1,2,3\} \big\}$. Their maximal common upper set~$\nested^\uparrow = \big\{ \{2\}, \{1,2,3\} \big\}$ is not an upper ideal face of~$\Nest(\building)$ since~$\lab(\{1,2,3\}, \nested^\uparrow) = \{1,3\}$ is not a singleton. Moreover, the face corresponding to~$\nested^\uparrow$ does not satisfy \SNLFP.
\end{remark}


\subsection{Diameter bounds}

Using Theorem~\ref{theo:increasingDiameter2} and Proposition~\ref{prop:upperSet2}, the lower bound on the diameter of the associahedron~\cite{Pournin}, the classical construction of graph associahedra of~\cite{CarrDevadoss,Postnikov} and the diameter of graphical zonotopes, we obtain the inequalities on the diameter~$\diam{\graphG}$ of~$\flipGraph(\graphG)$.

\begin{theorem}
\label{theo:boundsDiameter}
For any connected graph~$\graphG$ with~$n+1$ vertices and~$e$ edges, the diameter~$\diam{\graphG}$ of the flip graph~$\flipGraph(\graphG)$ is bounded by
\[
\max(e,2n - 18) \le \diam{\graphG} \le \binom{n+1}{2}.
\]
\end{theorem}

\begin{proof}
For the upper bound, we use that the diameter is non-decreasing (Theorem~\ref{theo:increasingDiameter2}) and that the $n$-dimensional permutahedron has diameter~$\binom{n+1}{2}$, the maximal number of inversions in a permutation of~$\fS_{n+1}$.

The lower bound consists in two parts. For the first part, we know that the normal fan of the graph associahedron~$\Asso(G)$ refines the normal fan of the graphical zonotope of~$\graphG$  (see \eg~\cite[Lect.~7]{Ziegler} for a reference on zonotopes). Indeed, the graph associahedron of~$\graphG$ can be constructed as a Minkowski sum of the faces of the standard simplex corresponding to tubes of~$\graphG$ (\cite{CarrDevadoss,Postnikov}) while the graphical zonotope of~$\graphG$ is the Minkowski sum of the faces of the standard simplex corresponding only to edges of~$\graphG$. Since the diameter of the graphical zonotope of~$\graphG$ is the number~$e$ of edges of~$\graphG$, we obtain that the diameter~$\diam{\graphG}$ is at least~$e$. For the second part of the lower bound, we use again Theorem~\ref{theo:increasingDiameter2} to restrict the argument to trees. Let~$\graphG[T]$ be a tree on $n+1$ vertices. We first discard some basic cases:
\begin{enumerate}[(i)]
\item If~$\graphG[T]$ has precisely two leaves, then~$\graphG[T]$ is a path and the graph associahedron~$\Asso(\graphG[T])$ is the classical $n$-dimensional associahedron, whose diameter is known to be larger than~$2n-4$ by L.~Pournin's result~\cite{Pournin}.
\item If~$\graphG[T]$ has precisely~$3$ leaves, then it consists in~$3$ paths attached by a $3$-valent node~$v$, see \fref{fig:decompositionTrees}\,(left). Let~$w$ be a neighbor of~$v$ and~$\pathG_1, \pathG_2$ denote the connected components of~$\graphG[T] \ssm w$. Observe that~$\pathG_1$ and~$\pathG_2$ are both paths and denote by~$p_1+1$ and~$p_2+1$ their respective lengths. Let~$\tubing'_1, \tubing''_1$ (resp.~$\tubing'_2, \tubing''_2$) be a diametral pair of maximal tubings on~$\pathG_1$ (resp.~on~$\pathG_2$), and consider the maximal tubings~$\tubing' = \tubing'_1 \cup \tubing'_2 \cup \{\pathG_1, \pathG_2\}$ and~$\tubing'' = \tubing''_1 \cup \tubing''_2 \cup \{\pathG_1, \pathG_2\}$ on the tree~$\graphG[T]$. Finally, denote by~$\tubing$ the maximal tubing on~$\graphG[T]$ obtained by flipping~$\pathG_1$ in~$\tubing'$. Since~$\{\pathG_1, \pathG_2\}$ is a common upper set to the inclusion posets of~$\tubing'$ and~$\tubing''$, Proposition~\ref{prop:upperSet2}\,(ii) ensures that there exists a geodesic from~$\tubing$ to~$\tubing''$ that starts by the flip from~$\tubing$ to~$\tubing'$. Moreover, Proposition~\ref{prop:upperSet2}\,(i) ensures that the distance between~$\tubing'$ and~$\tubing''$ is realized by a path staying in the face of~$\Asso(\graphG[T])$ corresponding to~$\{\pathG_1, \pathG_2\}$, which is the product of a classical $p_1$-dimensional associahedron by a classical $p_2$-dimensional associahedron. We conclude that
\[
\qquad \diam{\graphG[T]} \ge 1 + \diam{\pathG_1} + \diam{\pathG_2} \ge 1 + (2p_1 - 4) + (2p_2 - 4) = 2 (p_1 + p_2 + 2) - 11 = 2n - 11.
\]
\item If~$\graphG[T]$ has precisely~$4$ leaves, it either contains a single $4$-valent node~$v$ or precisely two $3$-valent nodes~$u,v$, see \fref{fig:decompositionTrees}\,(middle and right). Define~$w$ to be a neighbor of~$v$, not located in the path between~$u$ and~$v$ in the latter situation. Then~$w$ disconnects~$\graphG[T]$ into a path~$\pathG$ on~$p+1$ nodes and a tree~$\graphG[Y]$ with $y+1$ nodes and precisely~$3$ leaves. A similar argument as in~(ii) shows that
\[
\qquad \diam{\graphG[T]} \ge 1 + \diam{\pathG} + \diam{\graphG[Y]} \ge 1 + (2p - 4) + (2y - 11) = 2 (p + y + 2) - 18 = 2n - 18.
\]
\end{enumerate}
\begin{figure}
  \capstart
  \centerline{\includegraphics[scale=1]{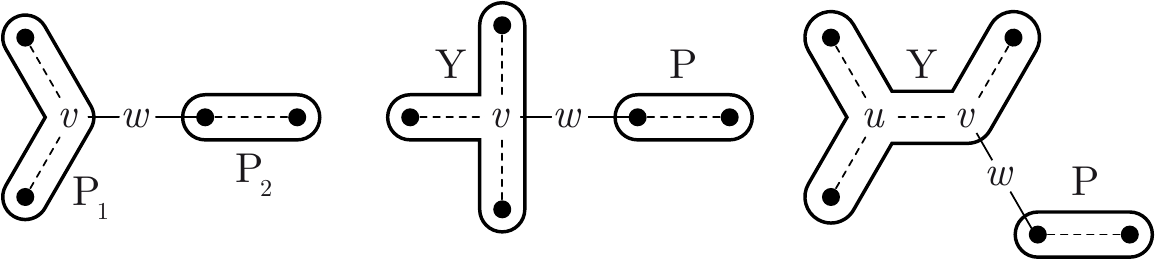}}
  \caption{Decompositions of trees with~$3$ or~$4$ leaves.}
  \label{fig:decompositionTrees}
\end{figure}

We can now assume that the tree~$\graphG[T]$ has~$k \ge 5$ leaves~$\leave_1, \dots, \leave_k$. Let~$\bar\ground=\ground \ssm \{\leave_1, \dots, \leave_k\}$ and ${\bar{\graphG[T]} = \graphG[T][\bar\ground]}$ denote the tree obtained by deletion of the leaves of~$\graphG[T]$. By induction hypothesis, there exists two maximal tubings~$\bar\tubing$ and~$\bar\tubing'$ on~$\bar{\graphG[T]}$ at distance at least~$2(n-k)-18$. Define~$\tube_i \eqdef \ground \ssm \{\leave_1, \dots, \leave_i\}$~for~${i \in [k]}$, and~$\tube'_j \eqdef \ground \ssm \{\leave_j, \dots, \leave_k\}$ for~$j \in [k]$. Consider the maximal tubings~${\tubing \eqdef \bar\tubing \cup \{\tube_1, \dots, \tube_k\}}$ and~${\tubing' \eqdef \bar\tubing' \cup \{\tube'_1, \dots, \tube'_k\}}$ on~$\graphG[T]$. We claim that the distance between these tubings is at least~${2n-18}$. To see it, consider the surjection~$\surjection$ from the tubings on~$\graphG[T]$ onto that of~$\bar{\graphG[T]} \sqcup \{\leave_1, \dots, \leave_k\}$ as defined in Section~\ref{subsec:nonDecreasingDiam}. It sends a path~$\tubing = \tubing_0, \dots, \tubing_\ell = \tubing'$ in the flip graph~$\flipGraph(\graphG[T])$ to a path 
\[
\bar\tubing \cup \{\{\leave_1\}, \dots, \{\leave_k\}\} = \surjection(\tubing_0), \dots, \surjection(\tubing_\ell) = \bar\tubing' \cup \{\{\leave_1\}, \dots, \{\leave_k\}\}
\]
 in the flip graph~$\flipGraph(\bar{\graphG[T]} \sqcup \{\leave_1, \dots, \leave_k\})$ with repeated entries. Since~$\bar\tubing$ and~$\bar\tubing'$ are at distance at least~$2(n-k)-18$ in the flip graph~$\flipGraph(\bar{\graphG[T]})$, this path has at least~$2(n-k)-18$ non-trivial steps, so we must show that it has at least~$2k$ repetitions. These repetitions appear whenever we flip a tube~$\tube_i$ or~$\tube'_j$. Indeed, we observe that the image~$\surjection(\tube)$ of any tube~$\tube \in \set{\tube_i}{i \in [k]} \cup \set{\tube'_j}{j \in [k]}$ is composed by~$\bar\ground$ together with single leaves of~$\graphG[T]$. Since all these tubes are connected components of~$\bar{\graphG[T]}$, we have~$\surjection(\tubing \ssm \{\tube\}) = \surjection(\tubing)$ for any maximal loaded tubing~$\tubing$ containing~$\tube$. To conclude, we distinguish three cases:
\begin{enumerate}[(i)]
\item If the tube~$\tube_k = \bar\ground = \tube'_1$ is never flipped along the path~$\tubing = \tubing_0, \dots, \tubing_\ell = \tubing'$, then we need at least~$\binom{k}{2}$ flips to transform~$\{\tube_1, \dots, \tube_k\}$ into~$\{\tube'_1, \dots, \tube'_k\}$. This can be seen for example from the description of the link of~$\tube_k$ in~$\nestedComplex(\graphG[T])$ in~\cite[Theorem~2.9]{CarrDevadoss}. Finally, we use that~$\binom{k}{2} \ge 2k$ since~$k \ge 5$.
\item Otherwise, we need to flip all~$\{\tube_1, \dots, \tube_k\}$ and then back all~$\{\tube'_1, \dots, \tube'_k\}$. If no flip of a tube~$\tube_i$ produces a tube~$\tube'_j$, we need at least~$2k$ flips which produces repetitions in~$\surjection(\tubing_0), \dots, \surjection(\tubing_\ell)$.
\item Finally, assume that we flip precisely once all~$\{\tube_1, \dots, \tube_k\}$ and then back all~$\{\tube'_1, \dots, \tube'_k\}$, and that a tube~$\tube_i$ is flipped into a tube~$\tube'_j$. According to the description of flips, we must have~${i = k-1}$ and~$j = 2$. If~$p \in [\ell]$ denotes the position such that~${\tubing_p \ssm \{\tube_{k-1}\} = \tubing_{p+1} \ssm \{\tube'_2\}}$, we moreover know that~$\tube_{k-1} \in \tubing_q$ for~$q \le p$, that~$\tube'_2 \in \tubing_q$ for~$q > p$, and that~${\bar\ground \in \tubing_p \cap \tubing_{p+1}}$. Applying the non-leaving-face property either to the upper set~$\{\tube_{k-1}, \tube_k\}$ in~$\Asso(\graphG{}[\tube_{k-1}])$ or to the upper set~$\{\tube'_1, \tube'_2\}$ in~$\Asso(\graphG{}[\tube'_2])$, we conclude that it would shorten the path~$\tubing_0, \dots, \tubing_\ell$ to avoid the flip of~$\tube_k = \bar\ground = \tube'_1$, which brings us back to Situation~(i). \qedhere
\end{enumerate}
\end{proof}

\begin{remark}
\label{rem:badLowerBound}
We note that although asymptotically optimal, our lower bound~$2n-18$ is certainly not sharp. We expect the correct lower bound to be the bound~$2n-4$ for the associahedron. Better upper bound can also be worked out for certain families of graphs. For example, L.~Pournin investigates the cyclohedra, \ie cycle associahedra. As far as trees are concerned, we understand better stars and their subdivisions. The diameter~$\diam{\completeG_{1,n}}$ for the star~$\completeG_{1,n}$ is exactly~$2n$ (for~$n \ge 5$), see \fref{fig:ctrexmNLFP}. In fact, the diameter of the graph associahedron of any starlike tree (subdivision of a star) on~$n+1$ vertices is bounded by~$2n$. To see it, we observe that any tubing is at distance at most~$n$ from the tubing~$\tubing_\circ$ consisting in all tubes adjacent to the central vertex. Indeed, we can always flip a tube in a tubing distinct from~$\tubing_\circ$ to create a new tube adjacent to the central vertex. This argument is not valid for non-starlike trees.
\end{remark}

\begin{remark}
The lower bound in Theorem~\ref{theo:boundsDiameter} shows that the diameter~$\diam{\graphG}$ is at least the number of edges of~$\graphG$. In view of Theorem~\ref{theo:increasingDiameter}, it is tempting to guess that the diameter~$\diam{\graphG}$ is of the same order as the number of edges of~$\graphG$. Adapting arguments from Remark~\ref{rem:badLowerBound}, we can show that the diameter of any tree associahedron~$\diam{\graphG[T]}$ is of order at most~$n \log n$. In any case, the following question remains open.
\end{remark}

\begin{question}
Is there a family of trees~$\graphG[T]_n$ on~$n$ nodes such that~$\diam{\graphG[T]_n}$ is of order~$n \log n$? Even more specifically, consider the family of trees illustrated in \fref{fig:3regularTree}: $\graphG[T]_1 = \graphG[K]_{1,3}$ (tripod) and~$\graphG[T]_{k+1}$ is obtained by grafting two leaves to each leaf of~$\graphG[T]_k$. What is the order of the diameter~$\diam{\graphG[T]_k}$?
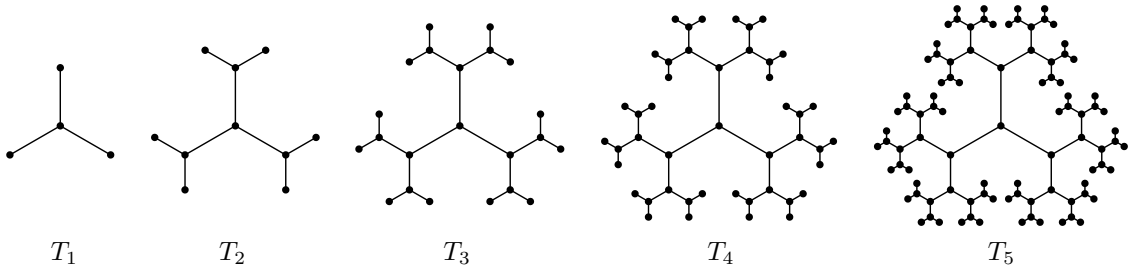
\begin{figure}[h]
  \capstart
  \centerline{
    \begin{overpic}[width=\textwidth]{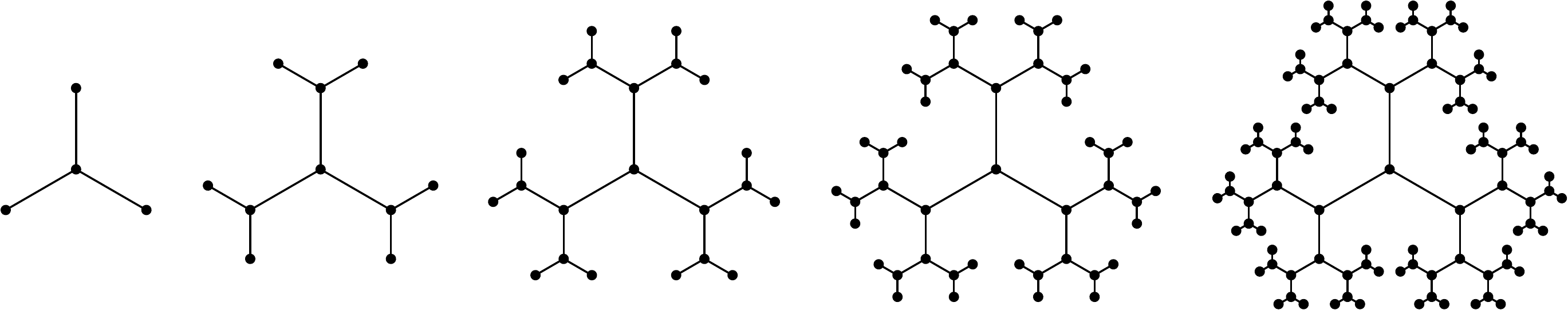}
  		\put(4,-3){$T_1$}
  		\put(19,-3){$T_2$}
  		\put(39,-3){$T_3$}
  		\put(62.5,-3){$T_4$}
  		\put(87.5,-3){$T_5$}
	\end{overpic}
  }
  \vspace{.3cm}
  \caption{The family of trees~$T_k$: the tree $T_1$ is the tripod and~$T_{k+1}$ is obtained from~$T_k$ by connecting two new nodes to each leaf of~$T_k$.}
  \label{fig:3regularTree}
\end{figure}
\end{question}

\begin{remark}
The upper bound~$\diam{\building} \le \binom{n+1}{2}$ holds for an arbitrary building set~$\building$ by Theorem~\ref{theo:increasingDiameter2} and the fact that the permutahedron is the nestohedron on the complete building set. In contrast, the lower bound is not valid for arbitrary connected building sets. For example, the nestohedron on the trivial connected building set~$\big\{ \{1\}, \dots, \{n+1\}, \{1, \dots, n+1\} \big\}$ is the $n$-dimensional simplex, whose diameter~is~$1$.
\end{remark}


\section{Hamiltonicity}
\label{sec:hamiltonicity}

In this section, we prove that the flip graph~$\flipGraph(\graphG)$ is Hamiltonian for any graph~$\graphG$ with at least~$2$ edges. This extends the result of H.~Steinhaus~\cite{Steinhaus}, S.~M.~Johnson~\cite{Johnson}, and H.~F.~Trotter~\cite{Trotter} for the permutahedron, and of J.~Lucas~\cite{Lucas} for the associahedron (see also~\cite{HurtadoNoy}). For all the proof, it is more convenient to work with spines than with tubings (remind Sections~\ref{subsec:spines} and~\ref{subsec:flips}). We first sketch the strategy of our proof.


\subsection{Strategy}
\label{subsec:strategy}

For any vertex~$v$ of~$\graphG$, we denote by~$\flipGraph_v(\graphG)$ the graph of flips on all spines on~$\graphG$ where~$v$ is a root. We call \defn{fixed-root subgraphs} of~$\flipGraph(\graphG)$ the subgraphs~$\flipGraph_v(\graphG)$ for~$v \in \ground$. Note that the fixed-root subgraph~$\flipGraph_v(\graphG)$ is isomorphic to the flip graph~$\flipGraph(\graphG{}[\hat v])$, where~$\graphG{}[\hat v]$ is the subgraph of~$\graphG$ induced by~$\hat v \eqdef \ground \ssm \{v\}$.

We now distinguish two extreme types of flips. Consider two maximal tubings~$\tubing, \tubing'$ on~$\graphG$ and tubes~$\tube \in \tubing$ and~$\tube' \in \tubing'$ such that~$\bar\tubing \eqdef \tubing \ssm \{\tube\} = \tubing' \ssm \{\tube'\}$. Let~$\spine, \spine'$ and~$\bar\spine$ denote the corresponding spines and~$\{v\} = \lab(\tube, \tubing)$ and~$\{v'\} = \lab(\tube', \tubing')$. We say that the flip~$\bar\tubing$ (or equivalently~$\bar\spine$) is
\begin{enumerate}[(i)]
\item a \defn{short flip} if both~$\tube$ and~$\tube'$ are singletons, that is, if $\{v,v'\}$ is a leaf of~$\bar\spine$;
\item a \defn{long flip} if~$\tube$ and~$\tube'$ are maximal proper tubes in~$\tubing$ and~$\tubing'$, that is, if~$\{v,v'\}$ is a root of~$\bar\spine$.
\end{enumerate}
Note that in a short flip, the vertices~$v,v'$ are necessarily adjacent in~$\graphG$. In the short flip~$\bar\spine$, we call \defn{short leaf} the leaf labeled by~$\{v,v'\}$ of~$\bar\spine$, \defn{short root} the root of the tree of~$\bar\spine$ containing the short leaf, and \defn{short child} the child~$w$ of the short root on the path to the short leaf. If the short leaf is already a child of the short root, then it coincides with the short child. Moreover, the short root, short child and short leaf all coincide if they form an isolated edge of~$\graphG$. In the long flip~$\bar\spine$, we call \defn{long root} the root labeled by~$\{v, v'\}$.

\enlargethispage{-.1cm}
We define a \defn{bridge} to be a square~$\bridge$ in the flip graph~$\flipGraph(\graphG)$ formed by two short and two long flips. We say that these two short (resp.~long) flips are~\defn{parallel}, and we borrow the terms long root and short leaf for the bridge~$\bridge$. \fref{fig:exmBridge} illustrates the notions of bridge, long flips and short flips.

\begin{figure}[b]
  \capstart
  \centerline{\includegraphics[width=\textwidth]{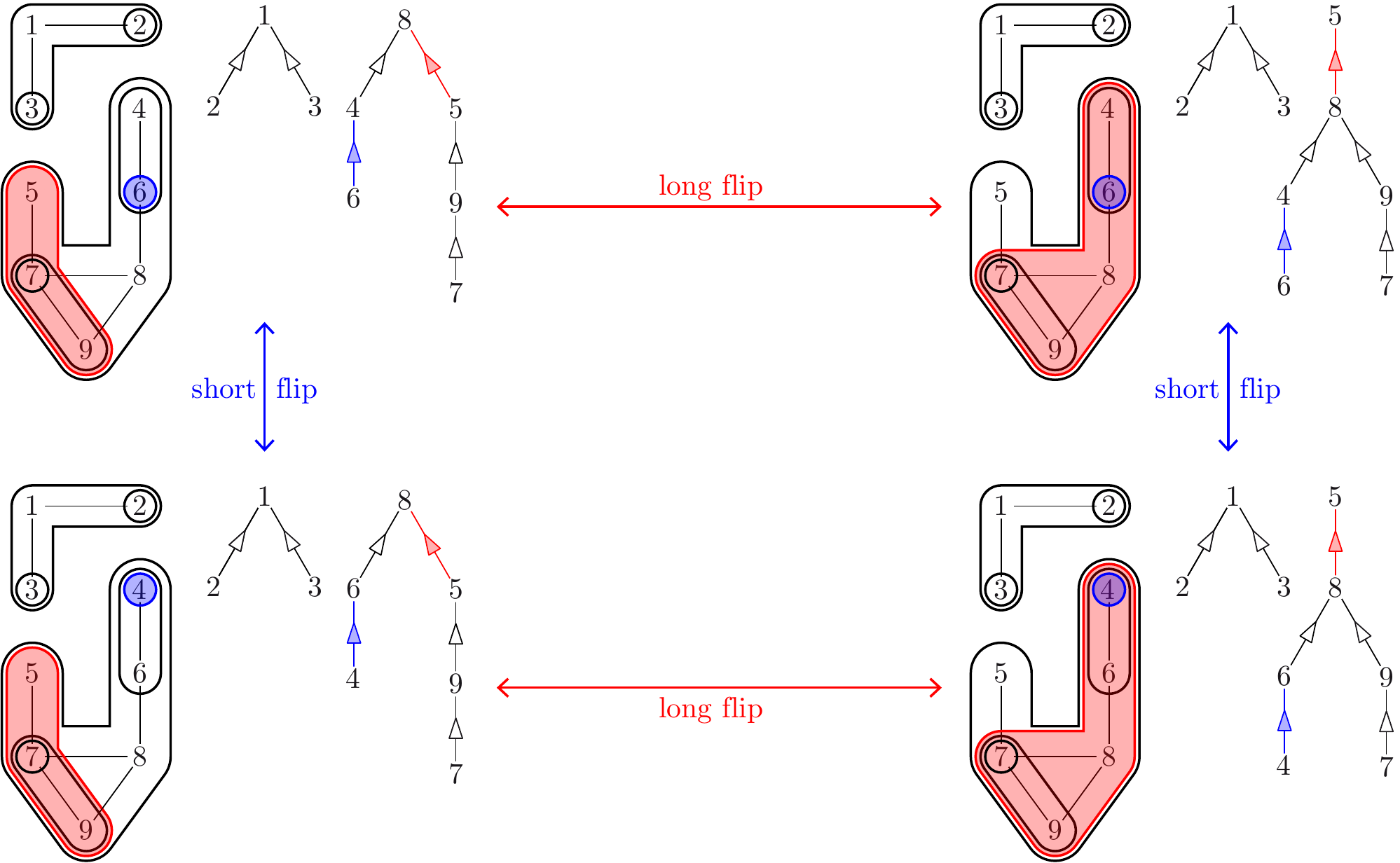}}
  \caption{A bridge, with two long flips (red) and two short flips (blue).}
  \label{fig:exmBridge}
\end{figure}

In terms of spines, a bridge can equivalently be defined as a spine~$\bridge$ of~$\graphG$ where all labels are singletons, except the label~$\{r,r'\}$ of a root and the label~$\{s,s'\}$ of a leaf. We denote by~$\bridge[r]$ the short flip of~$\bridge$ where~$r$ is a root, by~$\bridge[s]$ the long flip of~$\bridge$ where~$s$ is a leaf, and by~$\bridge[rs]$ the maximal spine on~$\graphG$ refining both~$\bridge[r]$ and~$\bridge[s]$, \ie where $r$ is a root and~$s$ a leaf. The flips~$\bridge[r']$ and~$\bridge[s']$ as well as the maximal spines~$\bridge[r's], \bridge[rs']$, and~$\bridge[r's']$ are defined similarly. These notations are summarized below

\bigskip
\centerline{\includegraphics[scale=.85]{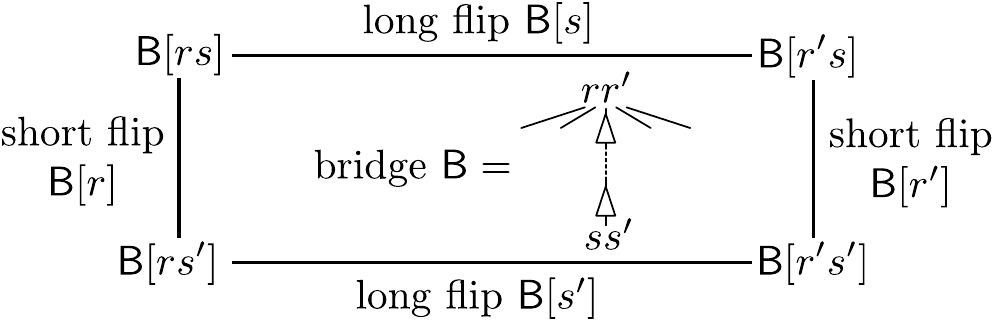}}

\bigskip
To obtain a Hamiltonian cycle~$\HamiltonianCycle$ of the flip graph~$\flipGraph(\graphG)$, we proceed as follows. The idea is to construct by induction a Hamiltonian cycle~$\HamiltonianCycle_{\hat v}$ of each flip graph~$\flipGraph(\graphG{}[\hat v])$, which is isomorphic to a Hamiltonian cycle~$\HamiltonianCycle_v$ in each fixed-root subgraph~$\flipGraph_v(\graphG)$. We then select an ordering~$v_1, \dots, v_{n+1}$ of~$\ground$,  such that two consecutive Hamiltonian cycles~$\HamiltonianCycle_{v_i}$ and~$\HamiltonianCycle_{v_{i+1}}$ meet the parallel short flips of a bridge~$\bridge_i$ for all~${i \in [n]}$. The Hamiltonian cycle of~$\flipGraph(\graphG)$ is then obtained from the union of the cycles~$\HamiltonianCycle_{v_1}, \dots, \HamiltonianCycle_{v_{n+1}}$ by exchanging the short flips with the long flips of all bridges~$\bridge_1, \dots, \bridge_n$, as illustrated in \fref{fig:strategy}.

\begin{figure}[h]
  \capstart
  \centerline{\includegraphics[width=.9\textwidth]{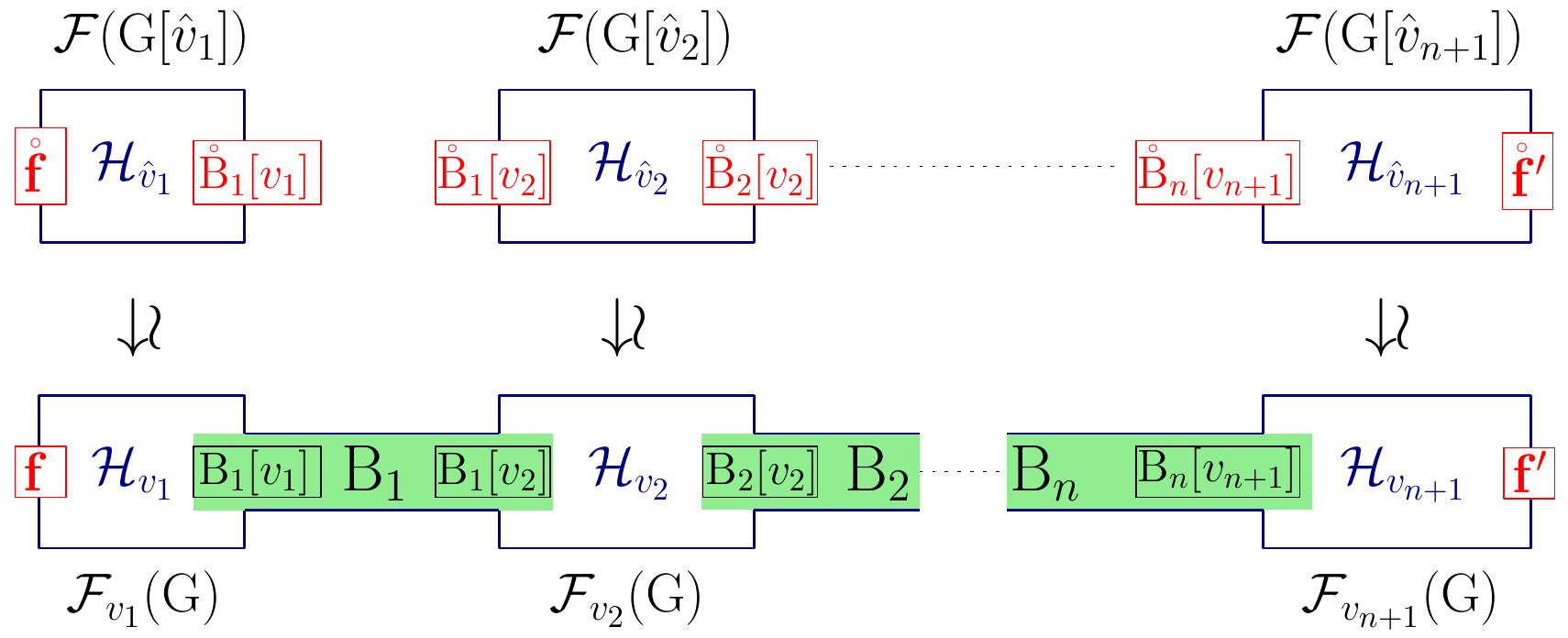}}
  \caption{The strategy for the proof of the hamiltonicity of the flip graph~$\flipGraph(\graphG)$. The circles above the short flips in the flip graphs~$\flipGraph(\graphG{}[\hat v_i])$ on top indicate that they are obtained by deleting the root~$v_i$ in the corresponding short flip of the fixed-root subgraph~$\flipGraph_{v_i}(\graphG)$ on bottom. See also Theorem~\ref{theo:Hamiltonian2}.}
  \label{fig:strategy}
\end{figure}

Of course, this description is a simplified and naive approach. The difficulty lies in that, given the Hamiltonian cycles~$\HamiltonianCycle_v$ of the fixed-root subgraphs~$\flipGraph_v(\graphG)$, the existence of a suitable ordering~$v_1, \dots, v_{n+1}$ of~$\ground$ and of the bridges~$\bridge_1, \dots, \bridge_n$ connecting the consecutive Hamiltonian cycles~$\HamiltonianCycle_{v_i}$ and~$\HamiltonianCycle_{v_{i+1}}$ is not guaranteed. To overpass this issue, we need to impose the presence of two forced short flips in each Hamiltonian cycle~$\HamiltonianCycle_v$. We include this condition in the induction hypothesis and prove the following sharper version of Theorem~\ref{theo:Hamiltonian}.

\begin{theorem}
\label{theo:Hamiltonian2}
For any graph~$\graphG$, any pair of short flips of~$\flipGraph(\graphG)$ with distinct short roots is contained in a Hamiltonian cycle of the flip graph~$\flipGraph(\graphG)$.
\end{theorem}

Note that for any graph~$\graphG$ with at least~$2$ edges, the flip graph~$\flipGraph(\graphG)$ always contains two short flips with distinct short roots. Theorem~\ref{theo:Hamiltonian} thus follows from the formulation of Theorem~\ref{theo:Hamiltonian2}.

The issue in our inductive approach is that the fixed-root subgraphs of~$\flipGraph(\graphG)$ do not always contain two edges, and therefore cannot be treated by Theorem~\ref{theo:Hamiltonian2}. Indeed, it can happen that:
\begin{itemize}
\item $\graphG{}[\hat v]$ has a single edge and thus the fixed-root subgraph~$\flipGraph_v(\graphG) \sim \flipGraph(\graphG{}[\hat v])$ is reduced to a single (short) flip. This case can still be treated with the same strategy: we consider this single flip~$\flipGraph_v(\graphG)$ as a degenerate Hamiltonian cycle and we can concatenate two bridges containing this short flip.
\item $\graphG{}[\hat v]$ has no edge and thus the fixed-root subgraph~$\flipGraph_v(\graphG) \sim \flipGraph(\graphG{}[\hat v])$ is a point. This is the case when~$\graphG$ is a star with central vertex~$v$ together with some isolated vertices. We need to make a special and independent treatment for this particular case. See Section~\ref{subsec:stars}.
\end{itemize}


\subsection{Disconnected graphs}
\label{subsec:disconnected}

We first show how to restrict the proof to connected graphs using some basic results on products of cycles. We need the following lemmas. 

\begin{lemma}
\label{lem:disconnected1}
For any two cycles~$\HamiltonianCycle, \HamiltonianCycle'$ and any two edges~$e,e'$ of~$\HamiltonianCycle \times \HamiltonianCycle'$, there exists a Hamiltonian cycle of~$\HamiltonianCycle \times \HamiltonianCycle'$ containing both~$e$ and~$e'$.
\end{lemma}

\begin{proof}
The idea is illustrated in~\fref{fig:cartesianProductCycles}. The precise proof is left to the reader.
\begin{figure}[h]
  \capstart
  \centerline{\raisebox{.2cm}{\includegraphics[scale=.7]{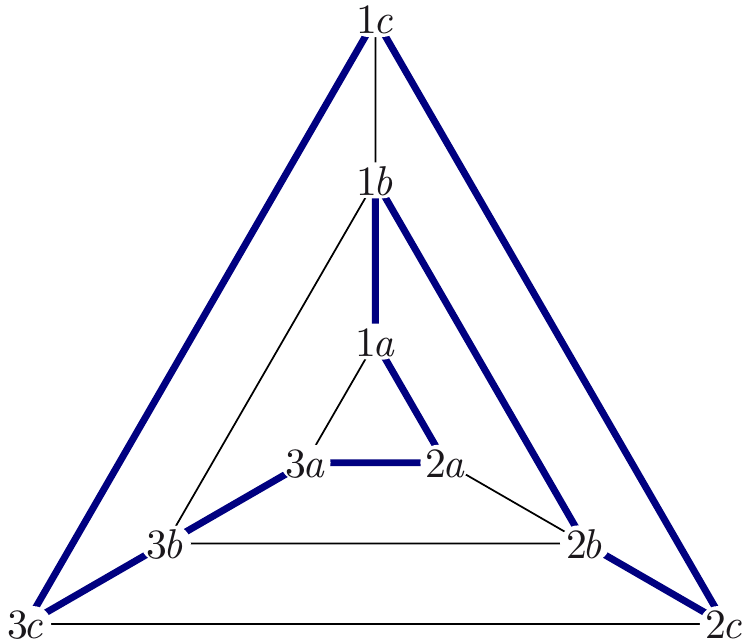}} \includegraphics[scale=.7]{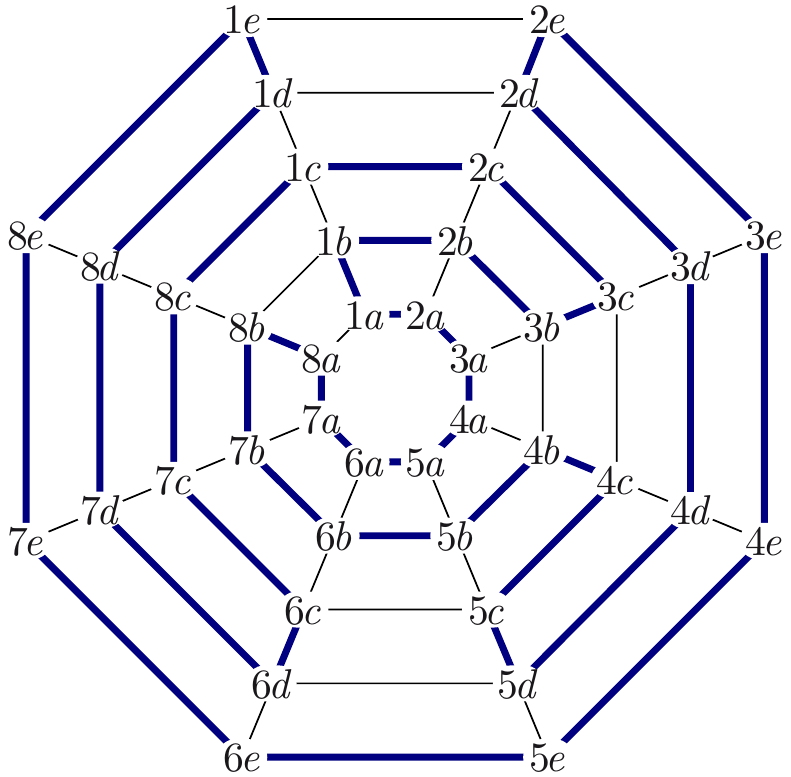}}
  \caption{An idea for the proof of Lemma~\ref{lem:disconnected1}. Any pair of edges is contained in a Hamiltonian cycle similar to those. The pictures represent Cartesian products of the cycle~$\HamiltonianCycle$ with the path obtained by deleting one edge in~$\HamiltonianCycle'$.}
  \label{fig:cartesianProductCycles}
\end{figure}
\end{proof}

\begin{lemma}
\label{lem:disconnected2}
For any cycle~$\HamiltonianCycle$, any isolated edge~$e_\circ$ and any two edges~$e,e'$ of~$\HamiltonianCycle \times e_\circ$, there exists a Hamiltonian cycle containing both~$e$ and~$e'$, as soon as one of the following conditions hold:
\begin{enumerate}
\item the edges~$e,e'$ are not both of the form~$\{v\} \times e_\circ$ with~$v \in \HamiltonianCycle$;
\item $e = \{v\} \times e_\circ$ and~$e' = \{v'\} \times e_\circ$ where~$\{v,v'\}$ is an edge of~$\HamiltonianCycle$;
\item $\HamiltonianCycle$ has an even number of edges.
\end{enumerate}
\end{lemma}

\begin{proof}
The idea is illustrated in~\fref{fig:cartesianProductCycleEdge}. The precise proof is left to the reader.
\begin{figure}[h]
  \capstart
  \centerline{\includegraphics[width=.7\textwidth]{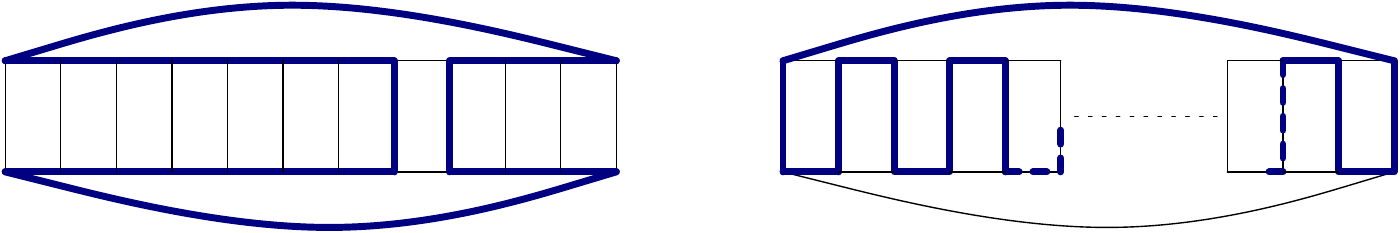}}
  \caption{An idea for the proof of Lemma~\ref{lem:disconnected2}. The right picture only works for even cycles.}
  \label{fig:cartesianProductCycleEdge}
\end{figure}
\end{proof}

\begin{corollary}
\label{coro:disconnected}
If two graphs~$\graphG, \graphG'$ both have the property that any pair of short flips of their flip graph with distinct short roots is contained in a Hamiltonian cycle of their flip graph, then~$\graphG \sqcup \graphG'$ fulfills the same property.
\end{corollary}

\begin{proof}
We have seen that the flip graph of the disjoint union of two graphs~$\graphG_1$ and~$\graphG_2$ is the product of their flip graphs~$\flipGraph(\graphG_1)$ and~$\flipGraph(\graphG_2)$. The statement thus follows from the previous lemmas.
\end{proof}


\subsection{Generic proof}
\label{subsec:proof}

We now present an inductive proof of Theorem~\ref{theo:Hamiltonian2}. Corollary~\ref{coro:disconnected} allows us to restrict to the case where~$\graphG$ is connected. For technical reasons, the stars and the graphs with at most~$6$ vertices will be treated separately. We thus assume here that~$\graphG$ is not a star and has at least $7$ vertices, which ensures that any fixed root subgraph of the flip graph~$\flipGraph(\graphG)$ has at least one short flip. Fix two short flips~$\shortFlip, \shortFlip'$ of~$\flipGraph(\graphG)$ with distinct short roots~$v_1, v_{n+1}$ respectively.

We follow the strategy described in Section~\ref{subsec:strategy} and illustrated in \fref{fig:strategy}.
To apply Theorem~\ref{theo:Hamiltonian2} by induction on~$\graphG{}[\hat v_i]$, the short flips~$\bridge_{i-1}[v_i]$ and~$\bridge_i[v_i]$ should have distinct short children.
This forbids certain positions for~$v_{i+1}$ in~$\bridge_{i-1}[v_i]$ illustrated in \fref{fig:conflict}, and motivates the following definition. We say that a vertex~$w$ and a short flip~$\shortFlip[g]$ with root~$v$ are \defn{in conflict} if either of the following happens:
\begin{enumerate}[(A)]
\item $\{w\}$ is the short child of~$\shortFlip[g]$ and all other children of~$v$ in~$\shortFlip[g]$ are isolated in~$\graphG{}[\hat v]$; \label{cond:conflictA}
\item the graph~$\graphG{}[\hat v]$ has at least three edges, the graph~$\graphG{}\ssm\{v,w\}$ has exactly one edge which is the short leaf of~$\shortFlip[g]$; \label{cond:conflictB}
\item the graph~$\graphG{}[\hat v]$ has exactly two edges, the graph~$\graphG{}\ssm\{v,w\}$ has exactly one edge which is the short leaf of~$\shortFlip[g]$, and $w$ is a child of~$v$. \label{cond:conflictC}
\end{enumerate}
It is immediate that a short flip is in conflict with at most one vertex. Observe also that if~$w$ is in the short leaf of~$\shortFlip[g]$, then $w$ and~$\shortFlip[g]$ cannot be in conflict.

\begin{figure}[h]
  \capstart
  \begin{overpic}[scale=1]{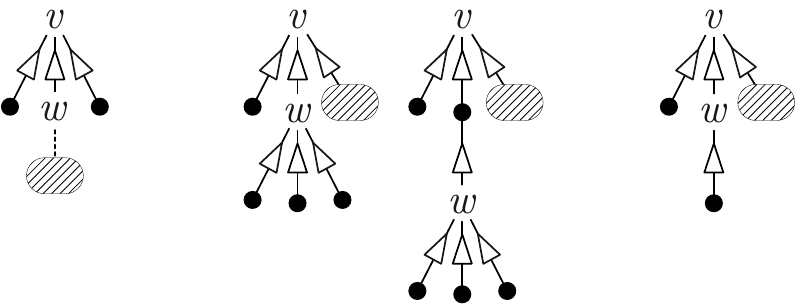}
  	\put(3,-6){(A)}
  	\put(45,-6){(B)}
  	\put(87,-6){(C)}
  \end{overpic}
  \bigskip
  \caption{Short flips in conflict with vertex~$w$. The short leaf is shaded. The second short flip of Case~(B) is in conflict with~$w$ only if the connected component of~$\graphG{}[\hat v]$ containing~$w$ is a star with central vertex~$w$.}
  \label{fig:conflict}
\end{figure}

\bigskip
We now show how we order the vertices~$v_1, \dots, v_{n+1}$ such that for each~$i \in [n]$ there exists a bridge~$\bridge_i$ connecting the fixed-root subgraphs~$\flipGraph_{v_i}(\graphG)$ and~$\flipGraph_{v_{i+1}}(\graphG)$.

\begin{lemma}
\label{lem:ordering}
There exists an ordering~$v_1, \dots, v_{n+1}$ of the vertices of~$\graphG$ (provided~$|\ground| \ge 7$) satisfying the following properties:
\begin{itemize}
\item $v_2$ and~$\shortFlip$ are not in conflict, and~$v_n$ and~$\shortFlip'$ are not in conflict, and
\item for any~$i \in [n]$, the graph~$\graphG$ contains an edge disjoint from~$\{v_i,v_{i+1}\}$.
\end{itemize}
\end{lemma}

\begin{proof}
Let~$a$ and~$a'$ denote the vertices in conflict with~$\shortFlip$ and~$\shortFlip'$ if any. Let~$D$ denote the set of \defn{totally disconnecting pairs} of~$\graphG$, \ie of pairs~$\{x,y\}$ such that~$\graphG \ssm \{x,y\}$ has no edge. We want to show that there exists an ordering on the vertices of~$\graphG$ in which neither~$\{v_1,a\}$ nor~$\{a',v_{n+1}\}$, nor any pair of~$D$ are consecutive. For this, we prove that if~$\graphG$ has at least~$5$ vertices and is not a star (\ie all edges contain a central vertex), then~$|D| \le 2$ and the pairs in~$D$ are not disjoint.

Suppose by contradiction that~$D$ contains two disjoint pairs~$\{x_1,y_1\}$ and~$\{x_2,y_2\}$. Then any edge of~$\graphG$ intersects both pairs, so that~$x_1,x_2,y_1,y_2$ are the only vertices in~$\graphG$ (by connectivity), contradicting that~$\graphG$ has at least~$5$ vertices. Suppose now that~$D$ contains three pairwise distinct pairs~$\{x,y_1\}$, $\{x,y_2\}$ and~$\{x,y_3\}$. Then any edge of~$\graphG$ contains~$x$ since it cannot contain~$y_1, y_2$ and~$y_3$ together. It follows that~$\graphG$ is a star with central vertex~$x$.

Since~$|D| \le 2$, at most $4$ pairs of vertices of~$\graphG$ cannot be consecutive in our ordering. It is thus clear that if there are enough other vertices, we can find a suitable ordering. In fact, it turns out that it is already possible as soon as~$\graphG$ has~$7$ vertices. It is easy to prove by a boring case analysis. We just treat the worst case below.

Assume that~$D = \{\{x,y\}, \{x,z\}\}$ where~$x,y,z \notin \{v_1, v_{n+1}\}$ and that~$x$ is in conflict with both short flips~$\shortFlip$ and~$\shortFlip'$. Since~$|\ground| \ge 7$, there exists two distinct vertices~${u,v \notin \{v_1, v_{n+1}, x, y, z\}}$ and we set~$v_2=z, v_3=y, v_4=u, v_5=x, v_6=v$ and choose any ordering for the remaining vertices. This order satisfies the requested conditions.
\end{proof}

\begin{remark}
\label{rem:ordering}
In fact, using similar arguments, one can easily check that the result of Lemma~\ref{lem:ordering} holds in the following situations:
\begin{itemize}
\item $|\ground| = 6$, and either $|D| \le 1$ or $D = \{\{x,y\}, \{x,z\}\}$ where $x$ is not in conflict with both~$\shortFlip, \shortFlip'$.
\item $|\ground| = 5$, and either $D = \varnothing$ or $D = \{\{x,y\}\}$ where neither~$x$ nor~$y$ is in conflict~with~both~$\shortFlip, \shortFlip'$.
\item $|\ground| = 5$, and $D=\{\{x,y\},\{x,z\}\}$ and~$|\{x,y,z\} \cap \{v_1,v_5\}| = 2$.
\item $|\ground| = 5$, and $D=\{\{x,y\},\{x,z\}\}$ and~$|\{x,y,z\} \cap \{v_1,v_5\}| = 1$ and neither of~$x,y,z$ is in conflict with any of~$\shortFlip$ and~$\shortFlip'$.
\end{itemize}
\end{remark}

Given such an ordering~$v_1, \dots, v_{n+1}$, we choose bridges~$\bridge_1, \dots, \bridge_n$ connecting the fixed-root subgraphs~$\flipGraph_{v_1}(\graphG), \dots, \flipGraph_{v_{n+1}}(\graphG)$. We start with the choice of~$\bridge_1$.

\begin{lemma}
\label{lem:horribleLemma}
There exists a bridge~$\bridge_1$ with root~$\{v_1,v_2\}$ such that
\begin{itemize}
\item if~$\flipGraph_{v_1}(\graphG)$ is a square, the short flips~$\shortFlip$ and~$\bridge_1[v_1]$ are distinct,
\item if~$\flipGraph_{v_1}(\graphG)$ is not reduced to a single flip nor to a square, the short flips~$\shortFlip$ and~$\bridge_1[v_1]$ have distinct short children,
\item $\bridge_1[v_2]$ and~$v_3$ are not in conflict, and
\item the singleton~$\{v_3\}$ is a child of~$v_2$ in~$\bridge_1[v_2]$ only if~$v_3$ is isolated in~$\graphG \ssm \{v_1, v_2\}$.
\end{itemize}
\end{lemma}

\begin{proof}
The proof is an intricate case analysis. In each case, we will provide a suitable choice for~$\bridge_1$, but the verification that this bridge exists and satisfies the conditions of the statement is immediate and left to the reader. We denote by~$\kappa$ the connected component of~$\graphG{}[\hat v_1]$ containing~$v_2$. The following cases cover all possibilities:
\begin{enumerate}[$\spadesuit$]
\item $\kappa = \{v_2\}$: 
	\begin{enumerate}[$\heartsuit$]
	\item $\graphG \ssm \{v_1, v_2\}$ has only one edge: the fixed root subgraph~$\flipGraph_{v_1}(\graphG)$ is reduced to the short flip~$\shortFlip$ and the bridge obtained by contracting~$\{v_1, v_2\}$ in~$\shortFlip$ suits for~$\bridge_1$.
	\item $\graphG \ssm \{v_1, v_2\}$ has at least two edges: we choose for~$\bridge_1$ any bridge with root~$\{v_1, v_2\}$ and with a short child different from that of~$\shortFlip$.
	\end{enumerate}
\item $\kappa \ne \{v_2\}$, so that~$\kappa$ has at least one edge:
	\begin{enumerate}[$\heartsuit$]
	\item $\graphG{}[\hat v_1] \ssm \kappa$ has no edge: Condition~\eqref{cond:conflictA} on~$\shortFlip$ and~$v_2$ ensures that~$v_2$ is not the short child of~$\shortFlip$. Since the short leaf of~$\shortFlip$ has to be in~$\kappa$, the short children of~$\shortFlip$ and~$\bridge_1[v_1]$ will automatically be different.
		\begin{enumerate}[$\diamondsuit$]
		\item $v_3 \notin \kappa$: any bridge with root~$\{v_1,v_2\}$ suits for~$\bridge_1$.
		\item $v_3 \in \kappa$:
			\begin{enumerate}[$\clubsuit$]
			\item $v_3$ is isolated in~$\kappa \ssm \{v_2\}$: any bridge with root~$\{v_1,v_2\}$ suits for~$\bridge_1$.
			\item $v_3$ is not isolated in~$\kappa \ssm \{v_2\}$: we choose for~$\bridge_1$ a bridge with root~$\{v_1,v_2\}$ and whose short leaf contains~$v_3$.
			\end{enumerate}
		\end{enumerate}
	\item $\graphG{}[\hat v_1] \ssm \kappa$ has precisely one edge~$e$:
		\begin{enumerate}[$\diamondsuit$]
		\item $e$ is not the short leaf of~$\shortFlip$: we choose for~$\bridge_1$ any bridge with root~$\{v_1,v_2\}$, short leaf~$e$ and in which~$\{v_3\}$ is a child of the root only if it is isolated in~$\graphG \ssm \{v_1,v_2\}$.
		\item $e$ is the short leaf of~$\shortFlip$:
			\begin{enumerate}[$\clubsuit$]
			\item $\kappa$ is a single edge: we choose for~$\bridge_1$ the bridge obtained by contracting~$\{v_1, v_2\}$ in the short flip opposite to~$\shortFlip$ in the square~$\flipGraph_{v_1}(\graphG)$ (which suits by Condition~\eqref{cond:conflictC}).
			\item $\kappa$ has at least two edges: Condition~\eqref{cond:conflictB} ensures that~$\kappa \ssm \{v_2\}$ has at least one edge.
				\begin{itemize}
				\item[$\circ$] $v_3 \notin \kappa$: any bridge with root~$\{v_1,v_2\}$ and short leaf in~$\kappa$ suits for~$\bridge_1$.
				\item[$\circ$] $v_3 \in \kappa$:
					\begin{itemize}
					\item[$\star$] $v_3$ is isolated in~$\kappa \ssm \{v_2\}$: any bridge with root~$\{v_1,v_2\}$ and short leaf in~$\kappa$ suits.
					\item[$\star$] $v_3$ is not isolated in~$\kappa \ssm \{v_2\}$: we choose for~$\bridge_1$ a bridge with root~$\{v_1,v_2\}$ and whose short leaf contains~$v_3$.
					\end{itemize}
				\end{itemize}
			\end{enumerate}
		\end{enumerate}
	\item $\graphG{}[\hat v_1] \ssm \kappa$ has at least two edges:
		\begin{enumerate}[$\diamondsuit$]
		\item $\graphG{}[\hat v_1] \ssm \kappa$ has only one non-trivial connected component: we choose for~$\bridge_1$ a bridge with root~$\{v_1, v_2\}$, with short leaf containing the non-isolated child of~$v_1$ in~$\shortFlip$ which is not in~$\kappa$, and in which~$\{v_3\}$ is a child of the root only if it is either isolated in~$\graphG \ssm \{v_1,v_2\}$ or the short child of~$\bridge_1[v_1]$.
		\item $\graphG{}[\hat v_1] \ssm \kappa$ has at least two non-trivial connected components: we choose for~$\bridge_1$ a bridge with root~$\{v_1, v_2\}$, with short leaf in a connected component of $\graphG{}[\hat v_1] \ssm \kappa$ not containing the short leaf of~$\shortFlip$, and in which~$\{v_3\}$ is a child of the root only if it is either isolated in~${\graphG \ssm \{v_1,v_2\}}$ or the short child of~$\bridge_1[v_1]$. \qedhere
		\end{enumerate}
	\end{enumerate}
\end{enumerate}
\end{proof}

The choice of~$\bridge_n$ is similar to that of~$\bridge_1$, replacing~$v_1, v_2, v_3$ and~$\shortFlip$ by~$v_{n+1}, v_n, v_{n-1}$ and~$\shortFlip'$ respectively.
For choosing the other bridges~$\bridge_2, \dots, \bridge_{n-1}$, we first observe the existence of certain special vertices in~$\graphG$.

We say that a vertex distinct from~$v_1$ and~$v_{n+1}$ which disconnects at most one vertex is an \defn{almost leaf} of~$\graphG$. Observe that~$\graphG$ contains at least one almost leaf: Consider a spanning tree~$\graphG[T]$ of~$\graphG$. If~$\graphG[T]$ is a path from~$v_1$ to~$v_{n+1}$, the neighbor of~$v_1$ in~$\graphG[T]$ is an almost leaf of~$\graphG$. Otherwise, any leaf of~$\graphG[T]$ distinct from~$v_1$ and~$v_{n+1}$ is an almost leaf of~$\graphG$.

Choose an almost leaf~$v_i$ of~$\graphG$ which disconnects no vertex if possible, and any almost leaf otherwise. We sequentially construct the bridges~$\bridge_2, \dots, \bridge_{i-1}$: once $\bridge_j$ is constructed, we choose~$\bridge_{j+1}$ using Lemma~\ref{lem:horribleLemma} where we replace~$v_1, v_2, v_3$ and~$\shortFlip$ by~$v_{j+1}, v_{j+2}, v_{j+3}$ and~$\bridge_j[v_{j+1}]$. Similarly, we choose the bridges~$\bridge_{n-1}, \dots, \bridge_{i+1}$: once $\bridge_{j+1}$ is constructed, we choose~$\bridge_j$ using Lemma~\ref{lem:horribleLemma} where we replace~$v_1, v_2, v_3$ and~$\shortFlip$ by~$v_{j+1}, v_{j}, v_{j-1}$ and~$\bridge_{j+1}[v_{j+1}]$. Note that the conditions on~$\bridge_1$ required in Lemma~\ref{lem:horribleLemma} ensure that the hypothesizes in Lemma~\ref{lem:ordering} can be propagated.

It remains to properly choose the last bridge~$\bridge_i$. This is done by the following statement.

\begin{lemma}
\label{lem:horribleLemma2}
Let~$\shortFlip[g], \shortFlip[h]$ be two short flips on~$\graphG$ with distinct roots~$v, w$ respectively. Assume that
\begin{enumerate}[(i)]
\item $\graphG \ssm \{v,w\}$ has at least one edge;
\item $\shortFlip[g]$ and~$w$ are not in conflict, and $\shortFlip[h]$ and~$v$ are not in conflict;
\item $\{v\}$ is a child of~$w$ in~$\shortFlip[h]$ only if~$v$ is isolated in~$\graphG{}[\hat w]$;
\item $v$ disconnects at most one vertex of~$\graphG$ and this vertex is not~$w$.
\end{enumerate}
Then there exists a bridge~$\bridge$ with root~$\{v,w\}$ such that~$\shortFlip[g]$ and~$\bridge[v]$ are distinct if~$\flipGraph_{v}[\graphG]$ is not reduced to a single flip and have distinct short children if~$\flipGraph_{v}[\graphG]$ is not a square, and similarly for~$\shortFlip[h]$ and~$\bridge[w]$.
\end{lemma}

\begin{proof}
Condition~(iv) implies that $\{w\}$ is the short child of~$\bridge[v]$ for any bridge~$\bridge$ with root~$\{v,w\}$. In contrast, Condition~(iv) and Condition~\eqref{cond:conflictA} for~$\shortFlip[g]$ and~$w$ ensure that~$\{w\}$ is not the short child of~$\shortFlip[g]$. Therefore, the conclusion of the lemma holds for~$\shortFlip[g]$ and~$\bridge[v]$, for any bridge~$\bridge$ with root~$\{v,w\}$. The difficulty is to choose~$\bridge$ in order to satisfy the conclusion for~$\shortFlip[h]$ and~$\bridge[w]$. For this, we distinguish various cases, in a similar manner as in Lemma~\ref{lem:horribleLemma}. Again, we provide in each case a suitable choice for~$\bridge$, but the verification that this bridge exists and satisfies the conditions of the statement is immediate and left to the reader.
\begin{enumerate}[$\spadesuit$]
\item $\graphG \ssm \{v,w\}$ has exactly one edge~$e$: this edge~$e$ has to be the short leaf of any bridge with root~$\{v,w\}$, thus Condition~\eqref{cond:conflictA} for~$\shortFlip[h]$ and~$v$ ensures that~$e$ is isolated in~$\graphG{}[\hat w]$.
	\begin{enumerate}[$\heartsuit$]
	\item $e$ is the short leaf of~$\shortFlip[h]$: Condition~\eqref{cond:conflictB} for~$\shortFlip[h]$ and~$v$ ensures that~$\flipGraph_w(\graphG)$ is either a single flip or a square (because~$v$ disconnects at most one vertex from~$\graphG$).
		\begin{enumerate}[$\diamondsuit$]
		\item $\flipGraph_w(\graphG)$ is a single flip: $\bridge$ is obtained by contracting~$\{v,w\}$ in~$\shortFlip[h]$.
		\item $\flipGraph_w(\graphG)$ is a square: Condition~\eqref{cond:conflictC} for~$\shortFlip[h]$ and~$v$ ensures that~$v$ is not a child of~$w$ in~$\shortFlip[h]$ and $\bridge$ is obtained by contracting~$\{v, w\}$ in the short flip opposite to~$\shortFlip[h]$ in the square~$\flipGraph_{w}(\graphG)$.
		\end{enumerate}
	\item $e$ is not the short leaf of~$\shortFlip[h]$: we choose for~$\bridge$ any bridge with root~$\{v,w\}$ and short leaf~$e$.
	\end{enumerate}
\item $\graphG \ssm \{v,w\}$ has at least two edges:
	\begin{enumerate}[$\heartsuit$]
	\item the short leaf and the short child of~$\shortFlip[h]$ coincide: any bridge with root~$\{v,w\}$ and a short leaf distinct from that of~$\shortFlip[h]$ suits for~$\bridge$.
	\item the short leaf and the short child of~$\shortFlip[h]$ are distinct: we choose for~$\bridge$ a bridge with root~$\{v,w\}$ whose short leaf contains the short child of~$\shortFlip[h]$.\qedhere
	\end{enumerate}
\end{enumerate}
\end{proof}

We have now chosen the order on the vertices~$v_1, \dots, v_{n+1}$ and chosen for each~$i \in [n]$ a bridge~$\bridge_i$ connecting the fixed-root subgraphs~$\flipGraph_{v_i}(\graphG)$ and~$\flipGraph_{v_{i+1}}(\graphG)$. Our choice forces the short flips~$\bridge_{i-1}[v_i]$ and~$\bridge_i[v_i]$ (as well as the short flips~$\shortFlip$ and~$\bridge_1[v_1]$ and the short flips~$\bridge_n[v_{n+1}]$ and~$\shortFlip'$) to be distinct if~$\flipGraph_{v_i}[\graphG]$ is not reduced to a single flip and have distinct short children if~$\flipGraph_{v_i}[\graphG]$ is not a square. We then construct a Hamiltonian cycle~$\HamiltonianCycle_i$ in each fixed-root subgraph~$\flipGraph_{v_i}(\graphG)$ such that $\HamiltonianCycle_1$ contains the short flips~$\shortFlip$ and~$\bridge_1[v_1]$,  $\HamiltonianCycle_{n+1}$ contains the short flips~$\bridge_n[v_{n+1}]$ and~$\shortFlip'$, and $\HamiltonianCycle_i$ contains the short flips~$\bridge_{i-1}[v_i]$ and~$\bridge_i[v_i]$ for all~$2 \le i \le n$. Note that
\begin{itemize}
\item when~$\flipGraph_{v_i}(\graphG)$ is reduced to a single flip, we just set~${\HamiltonianCycle_i = \flipGraph_{v_i}(\graphG)}$ and consider it as a degenerate Hamiltonian cycle;
\item when~$\flipGraph_{v_i}(\graphG)$ is a square, it is already a cyle;
\item otherwise, we apply Theorem~\ref{theo:Hamiltonian2} by induction to~$\graphG{}[\hat v_i]$ and obtain the Hamiltonian cycle~$\HamiltonianCycle_i$. The theorem applies since the short flips~$\bridge_{i-1}[v_i]$ and~$\bridge_i[v_i]$ have distinct short children, so that the corresponding short flips in~$\flipGraph(\graphG{}[\hat v_i])$ have distinct short roots. 
\end{itemize}
Finally, we obtain a Hamiltonian cycle of~$\flipGraph(\graphG)$ containing~$\shortFlip$ and~$\shortFlip'$ by gluing the cycles~$\HamiltonianCycle_1, \dots, \HamiltonianCycle_{n+1}$ together using the bridges~$\bridge_1, \dots, \bridge_n$ as explained in Section~\ref{subsec:strategy}. This is possible since the short flips~$\bridge_{i-1}[v_i]$ and~$\bridge_i[v_i]$ both belong to the Hamiltonian cycle~$\HamiltonianCycle_i$, and are distinct when~$\flipGraph_{v_i}(\graphG)$ is not reduced to a single flip. This concludes the proof for all generic cases. The remaining of the paper deals with the special cases of stars and graphs with at most $6$ vertices.


\subsection{Stars}
\label{subsec:stars}

We now treat the particular case of stars. Consider a ground set~$\ground$ where a vertex~$\ast$ is distinguished. The \defn{star} on~$\ground$ is the tree~$\starG_\ground$ where all vertices of~$\ground \ssm \{\ast\}$ are leaves connected to~$\ast$. The flip graph~$\flipGraph(\starG_\ground)$ has two kinds of fixed-root subgraphs:
\begin{itemize}
 \item $\flipGraph_\ast(\starG_\ground)$ is reduced to a single spine~$\circledast$ with root~$\ast$ and $n$ leaves;
 \item for any other vertex~$v \in \ground \ssm \{\ast\}$, the fixed-root subgraph~$\flipGraph_v(\starG_\ground)$ is isomorphic to the flip graph~$\flipGraph(\starG_{\hat v})$ of the star~$\starG_{\hat v}$, where~$\ast$ is still the distinguished vertex in~$\hat v = \ground \ssm \{v\}$. For a spine~$\spine \in \flipGraph_v(\starG_\ground)$, we denote by~$\tail{\spine}$ the unique subspine of~$\spine$, and we write in column~$\spine = {\substack{ v \\[.02cm] {\tail{\spine}}}}$.
\end{itemize}

To find a Hamiltonian cycle passing through forced short flips and through the spine~$\circledast$ we need to refine again the induction hypothesis of Theorem~\ref{theo:Hamiltonian2} as follows.

\begin{proposition}
\label{prop:HamiltonianCycleStars}
Assume that~$|\ground| \ge 3$, and fix two short flips~$\shortFlip, \shortFlip'$ of~$\flipGraph(\starG_\ground)$ with distinct roots~${r \ne r'}$ and a long flip~$\longFlip$ of~$\flipGraph(\starG_\ground)$ with root~$\{r'',\ast\}$. Then the flip graph~$\flipGraph(\starG_\ground)$ has a Hamiltonian cycle containing~$\shortFlip, \shortFlip', \longFlip$.
\end{proposition}

\begin{proof}
The proof works by induction on~$|\ground|$. If~$|\ground| = 3$, then~$\starG_\ground$ is a $3$-path and its flip graph is a pentagon. The case~$|\ground| = 4$ is solved by \fref{fig:star3} up to relabeling of~$\ground$. Namely, whatever triple~$\shortFlip, \shortFlip', \longFlip$ is imposed, there is a permutation of the leaves of~$\starG_{\{1,2,3,\ast\}}$ which sends the Hamiltonian cycle of \fref{fig:star3} to a Hamiltonian cycle passing through~$\shortFlip, \shortFlip', \longFlip$. Assume now that~$|\ground| \ge 5$. We distinguish two cases.

\begin{figure}[b]
  \capstart
  \centerline{\raisebox{1.4cm}{\includegraphics[width=.12\textwidth]{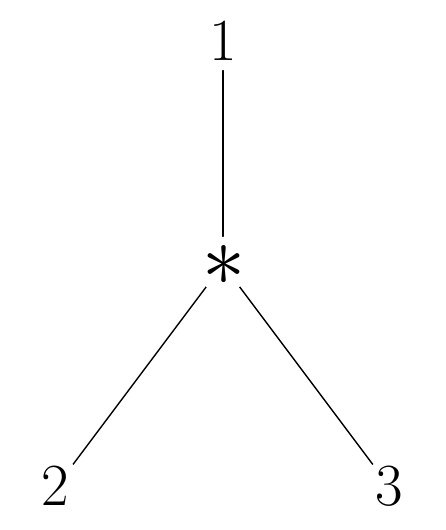}} \raisebox{.3cm}{\includegraphics[width=.5\textwidth]{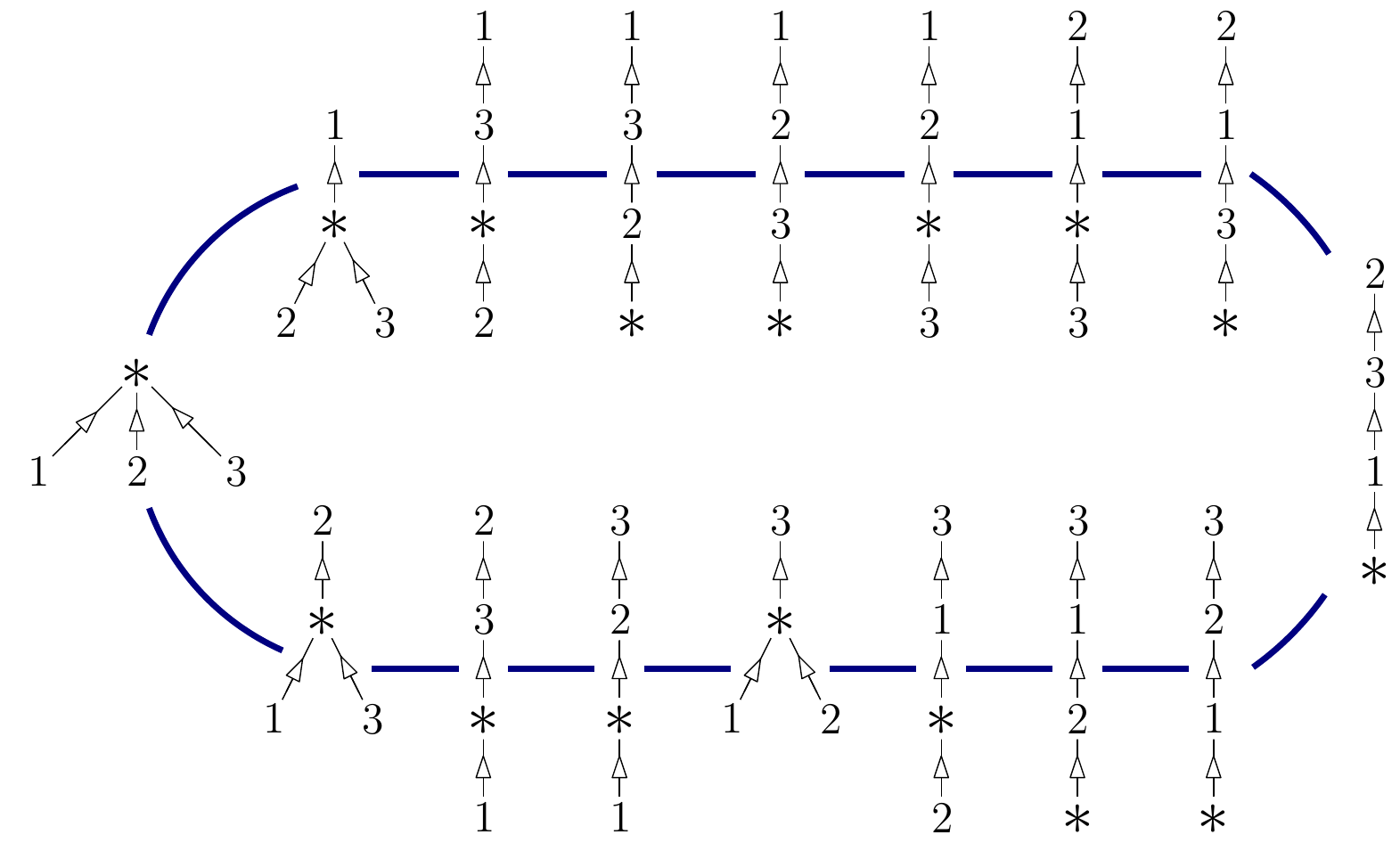}} \quad \includegraphics[width=.35\textwidth]{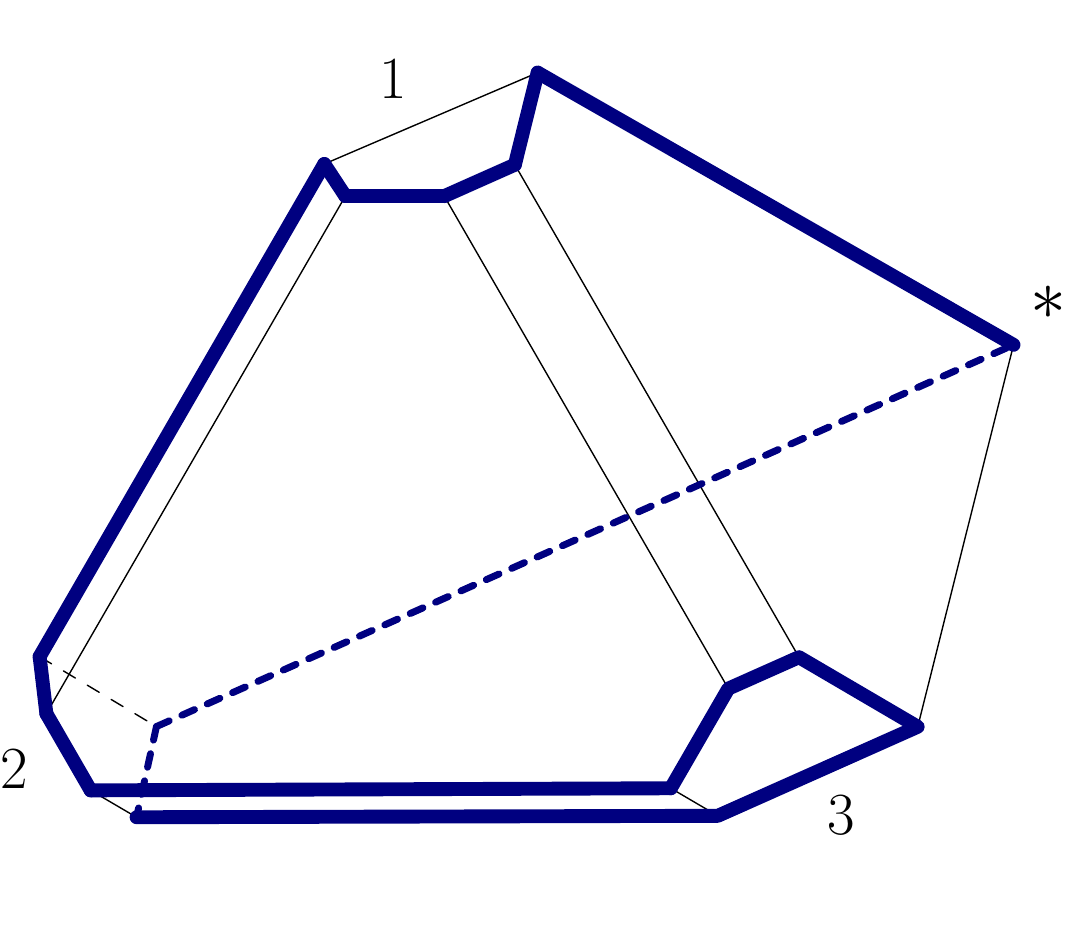}}
  \caption{A Hamiltonian cycle in the flip graph~$\flipGraph(\starG_{\{1, 2, 3, \ast\}})$. Up to permutations of the leaves~$\{1,2,3\}$, this cycle contains all possible triples~$\shortFlip, \shortFlip', \longFlip$ considered in Proposition~\ref{prop:HamiltonianCycleStars}.}
  \label{fig:star3}
\end{figure}

\para{Case 1:} $r'' \in \{r,r'\}$, say for instance~$r'' = r$. Let~$w'$ denote the child of~$r'$ in the short flip~$\shortFlip'$. Let~$v_1, \dots, v_{n-2}$ be an arbitrary ordering of~$\ground \ssm \{\ast, r, r'\}$ such that~$v_1 \ne w'$ (this is possible since~$|\ground| \ge 5$), and~$\bridge_1, \dots, \bridge_{n-2}$ any bridges such that the root of~$\bridge_i$ is~$\{v_{i-1}, v_i\}$ (where we set~$v_0 = r'$). We now choose inductively a Hamiltonian cycle~$\HamiltonianCycle_{\hat v}$ in each flip graph~$\flipGraph(\starG_{\hat v})$ for all~$v \in \ground \ssm \{\ast\}$ as follows.
\begin{enumerate}[(i)]
\item In~$\flipGraph(\starG_{\hat r})$, we choose a cycle~$\HamiltonianCycle_{\hat r}$ containing the short
        flip~$\tail{\shortFlip}$ and the long flip~$\circledast \flip \substack{r' \\ \!\! \circledast}$.
\item In~$\flipGraph(\starG_{\hat r'})$, we choose a cycle~$\HamiltonianCycle_{\hat r'}$ containing the short flips~$\tail{\shortFlip'}$ and~$\tail{\bridge}_1[r']$ and the long~flip~${\circledast \flip \substack{r \\ \circledast}}$.
\item In~$\flipGraph(\starG_{\hat v_i})$ for~$i \in [n-3]$, we choose a cycle~$\HamiltonianCycle_{\hat v_i}$ containing the short flips~$\tail{\bridge}_i[v_i]$ and~$\tail{\bridge}_{i+1}[v_i]$.
\item In~$\flipGraph(\starG_{\hat v_{n-2}})$, we choose a cycle~$\HamiltonianCycle_{\hat v_{n-2}}$ containing the short
        flip~$\tail{\bridge}_{n-2}[v_{n-2}]$.
\end{enumerate}

Note that these Hamiltonian cycles exist by induction hypothesis. Indeed, the short flips~$\tail{\bridge}_i[v_i]$ and~$\tail{\bridge}_{i+1}[v_i]$ have distinct roots~$v_{i-1}$ and~$v_{i+1}$. The only delicate case is thus Point~(ii): the short flips~$\tail{\shortFlip'}$ and~$\tail{\bridge}_1[r']$ have distinct roots since we forced~$v_1$ to be different from~$w'$. Each Hamiltonian cycle~$\HamiltonianCycle_{\hat v}$ on~$\flipGraph(\starG_{\hat v})$ induces a Hamiltonian cycle~$\HamiltonianCycle_v$ on~$\flipGraph_v(\starG_\ground)$ (just add~$v$ at the root in all spines). From these Hamiltonian cycles, we construct a Hamiltonian cycle for~$\flipGraph(\starG_\ground)$ as illustrated in~\fref{fig:herediteStar1}. We join~$\HamiltonianCycle_r$ with~$\HamiltonianCycle_{r'}$ by deleting the flips~${\substack{r \\ \circledast} \flip \substack{\!\! r \\ r' \\ \!\! \circledast}}$ and~${\substack{r' \\ \!\! \circledast} \flip \substack{r' \\ \!\! r \\ \!\! \circledast}}$ while inserting the long flips~$\substack{r \\ \circledast} \flip \circledast \flip \substack{r' \\ \!\! \circledast}$ and~$\substack{\!\! r \\ r' \\ \!\! \circledast} \flip \substack{r' \\ \!\! r \\ \!\! \circledast}$. Finally, we use the bridges~$\bridge_1, \dots, \bridge_{n-2}$ to connect the resulting cycle to the cycles~$\HamiltonianCycle_{v_1}, \dots, \HamiltonianCycle_{v_{n-2}}$ by exchanging their short flips with their long flips.

\begin{figure}[h]
  \capstart
  \centerline{\includegraphics[width=.9\textwidth]{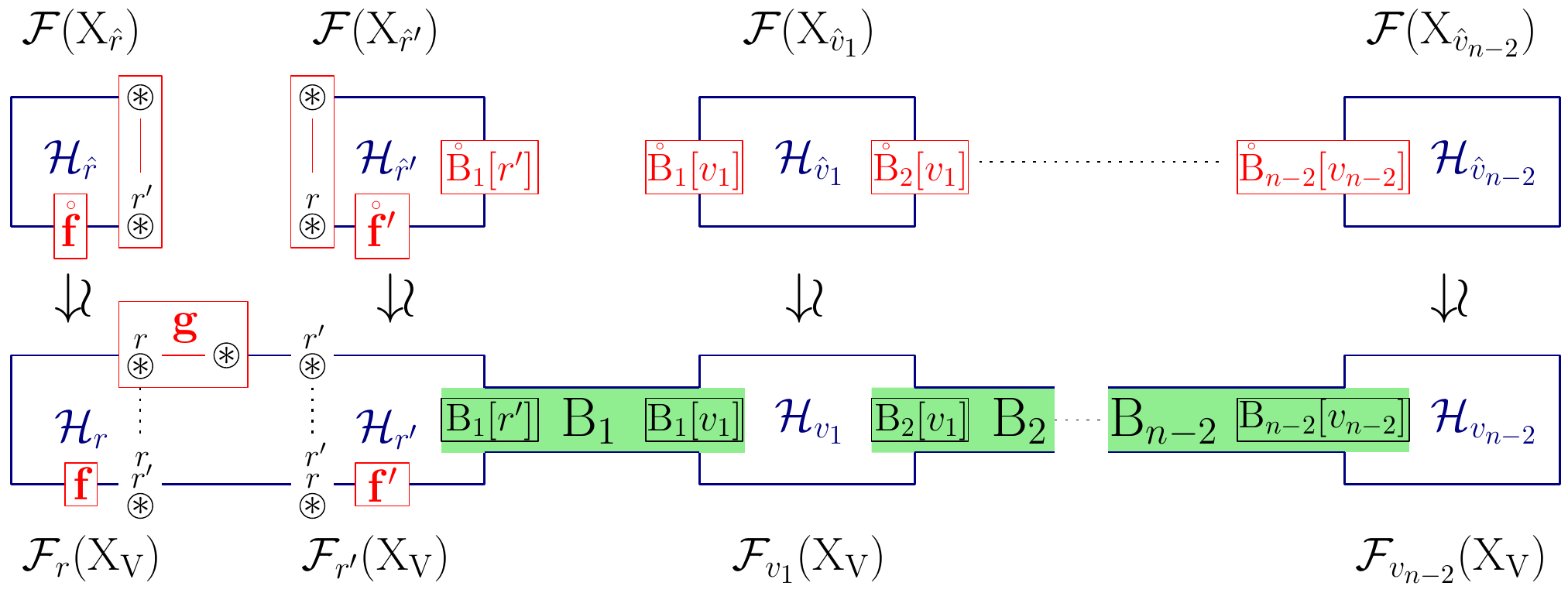}}
  \caption{Construction of a Hamiltonian cycle in~$\flipGraph(\starG_\ground)$ when~$r'' = r$.}
  \label{fig:herediteStar1}
\end{figure}

\para{Case 2:} $r'' \notin \{r,r'\}$. Let~$v_1, \dots, v_{n-3}$ be an arbitrary ordering of~${\ground \ssm \{\ast, r, r', r''\}}$, and~$\bridge_1, \dots, \bridge_{n-3}$ any bridges such that the root of~$\bridge_i$ is~$\{v_{i-1}, v_i\}$ (where we set~$v_0 = r''$). We now choose inductively a Hamiltonian cycle~$\HamiltonianCycle_{\hat v}$ in each flip graph~$\flipGraph(\starG_{\hat v})$ for all~$v \in \ground \ssm \{\ast, r'\}$ as follows.
\begin{enumerate}[(i)]
\item In~$\flipGraph(\starG_{\hat r})$, we choose a cycle~$\HamiltonianCycle_{\hat r}$ containing the short flip~$\tail{\shortFlip}$ and the long flip~$\circledast \flip \substack{r'' \\ \!\! \circledast}$.
\item In~$\flipGraph(\starG_{\hat r''})$, we choose a cycle~$\HamiltonianCycle_{\hat r''}$ containing a short flip~$\tail{\mathbf{h}}$ with root~$r'$, the short flip~$\tail{\bridge}_1[r'']$ and the long~flip~${\circledast \flip \substack{r \\ \circledast}}$.
\item In~$\flipGraph(\starG_{\hat v_i})$ for~$i \in [n-4]$, we choose a cycle~$\HamiltonianCycle_{\hat v_i}$ containing the short flips~$\tail{\bridge}_i[v_i]$ and~$\tail{\bridge}_{i+1}[v_i]$.
\item In~$\flipGraph(\starG_{\hat v_{n-3}})$, we choose a cycle~$\HamiltonianCycle_{\hat v_{n-3}}$ containing the short flip~$\tail{\bridge}_{n-3}[v_{n-3}]$ and a short flip~$\tail{\mathbf{k}}$ with root~$r'$.
\end{enumerate}
\begin{figure}[h]
  \capstart
  \centerline{\includegraphics[width=.9\textwidth]{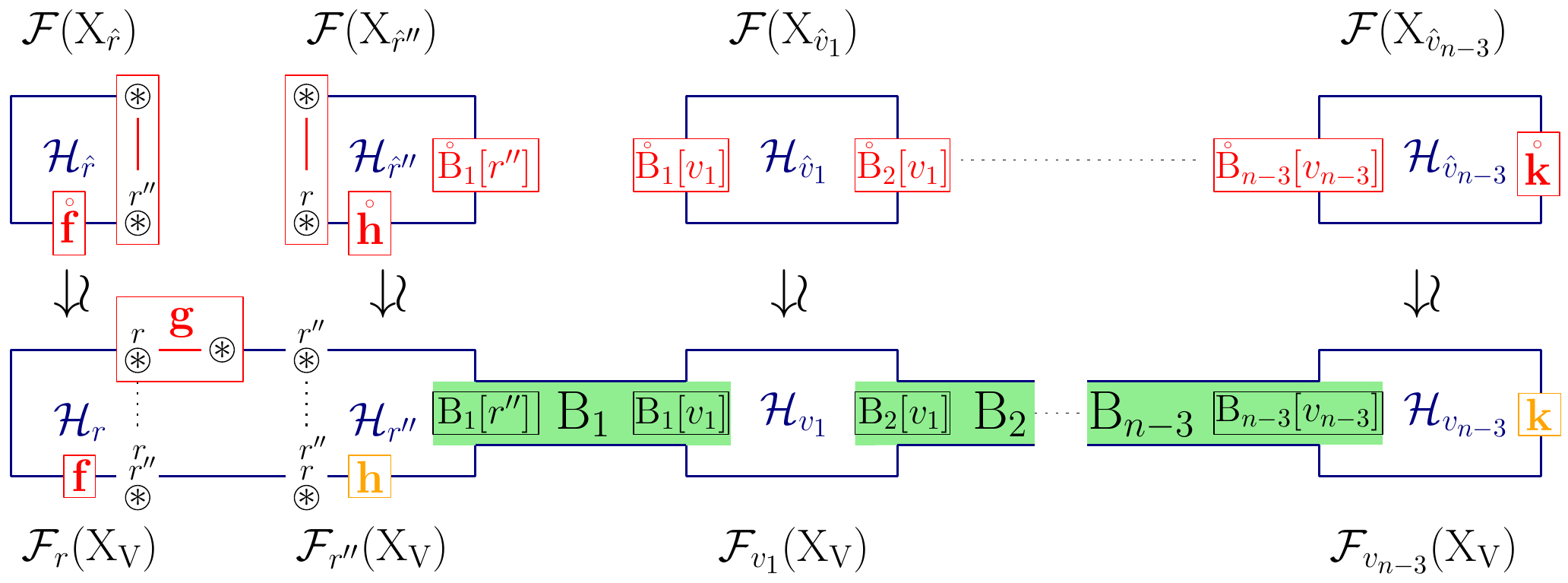}}
  \caption{Construction of a Hamiltonian cycle in~$\flipGraph(\starG_\ground)$ when~$r'' \notin \{r, r'\}$.}
  \label{fig:herediteStar2}
\end{figure}
Each Hamiltonian cycle~$\HamiltonianCycle_{\hat v}$ on~$\flipGraph(\starG_{\hat v})$ induces a Hamiltonian cycle~$\HamiltonianCycle_v$ on~$\flipGraph_v(\starG_\ground)$ (just add~$v$ at the root in all spines). From these Hamiltonian cycles, we construct the cycle illustrated in \fref{fig:herediteStar2}. We still have to enlarge this cycle to cover~$\flipGraph_{r'}(\starG_\ground)$. Let~$\mathbf{h}'$ and~$\mathbf{k}'$ denote the short flips in~$\flipGraph_{r'}(\starG_\ground)$ parallel to the short flips~$\mathbf{h}$ and~$\mathbf{k}$ respectively. Since~$r'' \ne v_{n-3}$, the root~$w'$ of~$\tail{\shortFlip'}$ cannot coincide with both. Assume for example that~$w' \ne r''$. By induction, we can then find a Hamiltonian cycle~$\HamiltonianCycle_{\hat r'}$ of~$\flipGraph(\starG_{\hat r'})$ containing both~$\tail{\shortFlip'}$ and~$\tail{\mathbf{h}'}$. This cycle induces a Hamiltonian cycle~$\HamiltonianCycle_{r'}$ of~$\flipGraph_{r'}(\starG_\ground)$ passing through~$\shortFlip'$ and~$\mathbf{h}'$. We can then connect this cycle to the cycle of \fref{fig:herediteStar2} by exchanging the parallel short flips~$\mathbf{h}$ and~$\mathbf{h}'$ by the corresponding parallel long flips. In the situation when~$w' = r''$, we have~$w' \ne v_{n-3}$ and we argue similarly by attaching~$\flipGraph_{r'}(\starG_\ground)$ to~$\mathbf{k}$ instead of~$\mathbf{h}$.
\end{proof}


\subsection{Graph with at most~$6$ vertices}
\label{subsec:6vertices}

Again we will focus on connected graphs because of Corollary~\ref{coro:disconnected}. The analysis for graphs with at most~$3$ vertices is immediate. We now treat separately the graphs with~$4,5$ and~$6$ vertices, which are not stars (stars have been treated in the previous section).

\subsubsection{Graphs with $4$ vertices}

We consider all possible connected graphs on $4$ vertices and exhibit explicit Hamiltonian cycles of their flip graphs. To do so, we could draw a cycle of spines as in \fref{fig:star3}\,(middle). Instead, we rather draw the Hamiltonian cycle on the flip graph~$\flipGraph(\graphG)$ represented as the $1$-skeleton of the graph associahedron~$\Asso(\graphG)$ as in \fref{fig:star3}\,(right). Let us remind from~\cite{CarrDevadoss} that the graph associahedron~$\Asso(\graphG)$ is obtained from the standard simplex~$\triangle_{\ground} \eqdef \conv \set{e_v}{v \in \ground}$ (where~$(e_v)_{v \in \ground}$ denotes the canonical basis of~$\R^\ground$) by successive truncations of the faces ${\triangle_{\ground \ssm \tube} = \conv\set{e_v}{v \in \ground \ssm \tube}}$ for the tubes~$\tube$ of~$\graphG$, in decreasing order of dimension. Each tube~$\tube$ of~$\graphG$ corresponds to a facet~$\face_{\tube}$ of~$\Asso(\graphG)$, and each maximal tubing~$\tubing$ corresponds to the vertex of~$\Asso(\graphG)$ which belongs to all facets~$\face_{\tube}$ for~$\tube \in \tubing$. In~\fref{fig:vertexSpineCorrespondance}\,(right), we label the positions of the vertices of~$\triangle_{\ground}$ before the truncations. The fixed-root subgraphs appear as the $1$-skeleta of the four shaded faces of~$\graphG$, and the bridges are the five thin parallelograms (the short flips correspond to their short sides, and the long flips correspond to their long sides).

\begin{figure}[h]
  \capstart
  \centerline{\includegraphics[width=.75\textwidth]{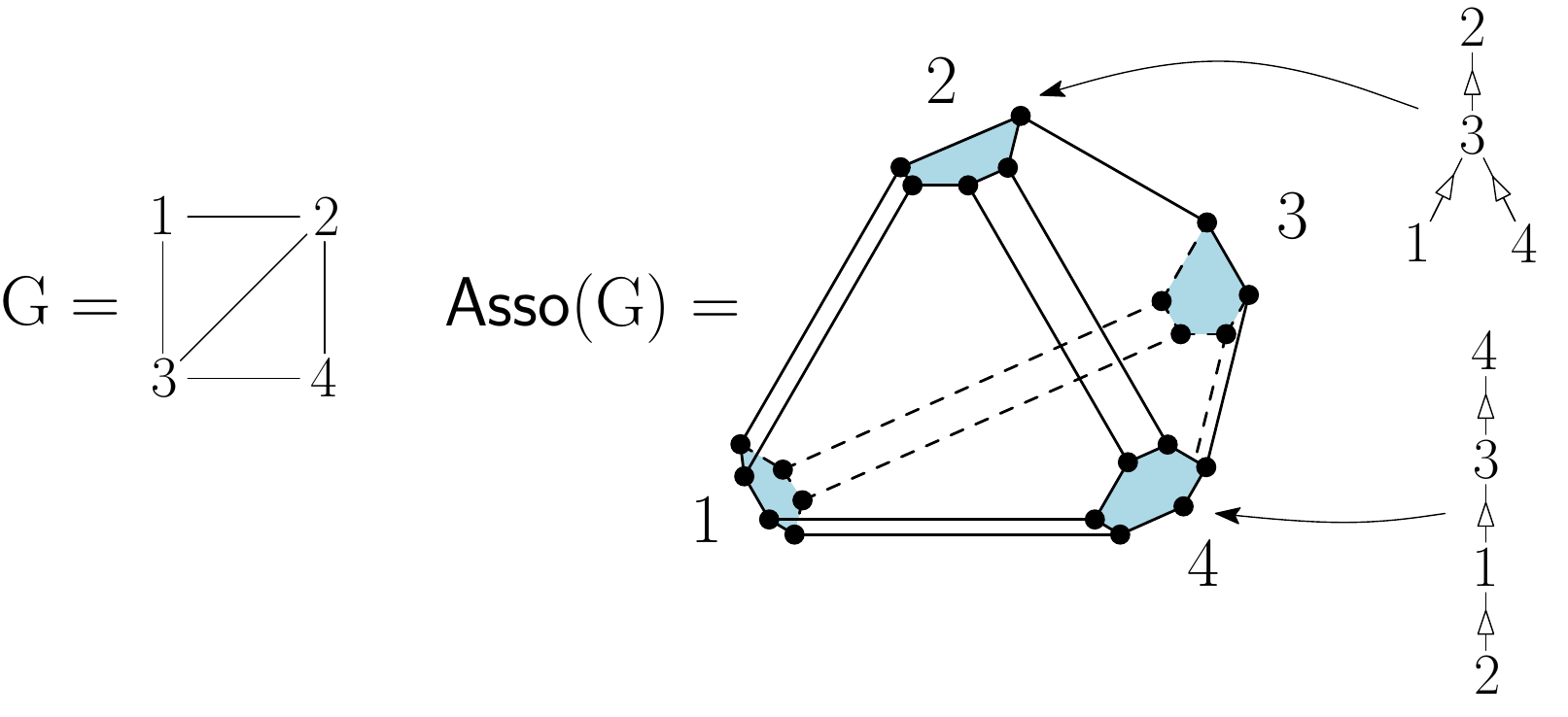}}
  \vspace{-.5cm}
  \caption{Correspondence between vertices of~$\Asso(\graphG)$ and spines on~$\graphG$.}
  \label{fig:vertexSpineCorrespondance}
\end{figure}

Using these conventions, \fref{fig:fourVertices} represents Hamiltonian cycles for the flip graphs on all connected graphs on $4$ vertices (the $4$-star was already treated in~\fref{fig:star3}). The Hamiltonian cycles, together with their orbits under the action of the isomorphism group of the corresponding graph, prove the following statements, which imply Theorem~\ref{theo:Hamiltonian2} for all graphs on $4$ vertices.

\begin{proposition}
\label{prop:4vertices}
\begin{enumerate}[(a)]
\item For any graph~$\graphG$ on at most~$4$ vertices, any pair of short flips (even with the same root) is contained in a Hamiltonian cycle of~$\flipGraph(\graphG)$. \label{cond:bonusa}
\item For the stars on~$3$ and~$4$ vertices, each triple consisting of two short flips (even with the same root) and one long flip as in Proposition~\ref{prop:HamiltonianCycleStars} is contained in a Hamiltonian cycle of~$\flipGraph(\graphG)$. \label{cond:bonusb}
\item For the classical~$3$-dimensional (path) associahedron, there exists a Hamiltonian cycle containing simultaneously all short flips. \label{cond:bonusc}
\item For all connected graphs on~$4$ vertices, there exist a Hamiltonian cycle of~$\flipGraph(\graphG)$ containing at least one short flip in each fixed-root subgraph. We can even preserve this property if we impose the Hamiltonian cycle to pass through one distinguished short flip. \label{cond:bonusd}
\end{enumerate}
\end{proposition}

\begin{figure}
  \capstart
  \centerline{\includegraphics[width=\textwidth]{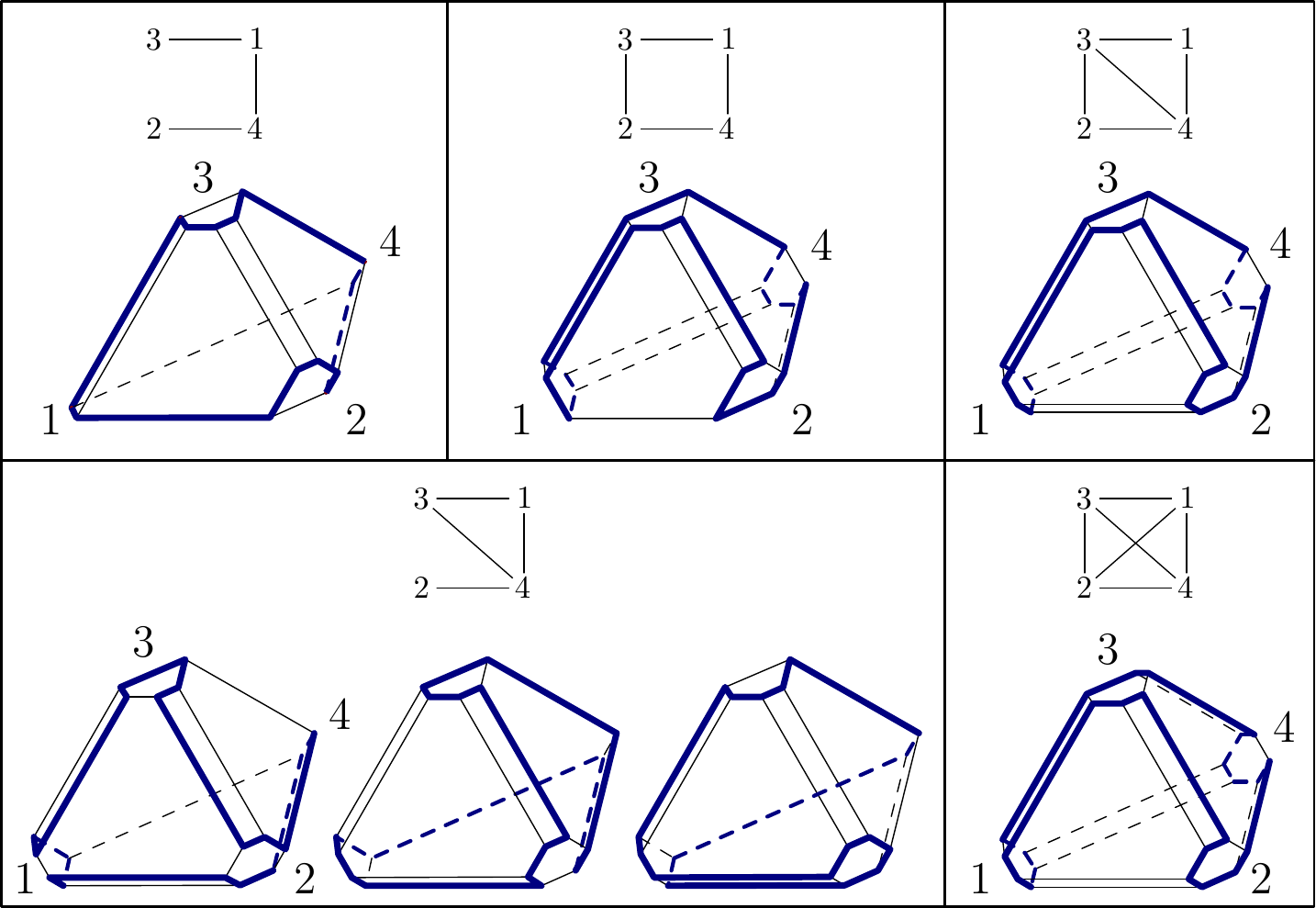}}
  \caption{Hamiltonian cycles showing Proposition~\ref{prop:4vertices}. Each vertex of the graph associahedra corresponds to a spine as explained in \fref{fig:vertexSpineCorrespondance}.}
  \label{fig:fourVertices}
\end{figure}

\subsubsection{Graphs with 5 vertices}

Graphs on $5$ vertices are treated by a case analysis. As in the proof of Lemma~\ref{lem:ordering}, we will denote by~$D$ the set of totally disconnecting pairs of~$\graphG$, \ie pairs~$\{x, y\}$ of vertices of~$\graphG$ such that~$\graphG \ssm \{x,y\}$ has no edge. Recall from the proof of Lemma~\ref{lem:ordering} that~$D$ has at most two elements and that they are not disjoint.

Consider now a graph~$\graphG$ on~$5$ vertices. According to Remark~\ref{rem:ordering}, the proof of Section~\ref{subsec:proof} applies in various configurations. We treat here the remaining cases. As we observed in Proposition~\ref{prop:4vertices}\,\eqref{cond:bonusa} that for any connected graph~$\graphG$ on at most $4$ vertices, any pair of short flips (even with the same root) is contained in a Hamiltonian cycle of~$\flipGraph(\graphG)$, we can ignore Condition~\eqref{cond:conflictA} in the definition of conflict. We therefore say that a vertex~$w$ and a short flip~$\shortFlip[g]$ with root~$v$ are \defn{in conflict} if~$\graphG \ssm \{v,w\}$ has a single edge which is the short leaf of~$\shortFlip[g]$, and~$w$ is a child of~$v$. With this definition, there is only one bridge connecting~$\flipGraph_v(\graphG)$ and~$\flipGraph_w(\graphG)$, but we cannot use it if we want the short flip~$\shortFlip[g]$ to belong to the Hamiltonian cycle. One can check that the conclusions of Lemmas~\ref{lem:horribleLemma} and~\ref{lem:horribleLemma2} still hold in this situation.

We first suppose that~$D=\{\{x,y\}\}$ is a singleton and that either~$x$ or~$y$ is in conflict with both~$\shortFlip$ and~$\shortFlip'$. Checking all connected graphs on five vertices, we see that this situation can only happen for the following graphs:
\begin{center}
$\graphG_1 = \;$\raisebox{-1cm}{\includegraphics[scale=1]{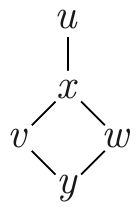}} \qquad $\graphG_2 = \;$\raisebox{-1cm}{\includegraphics[scale=1]{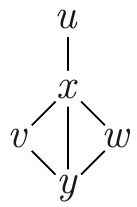}} \qquad $\graphG_3 = \;$\raisebox{-1cm}{\includegraphics[scale=1]{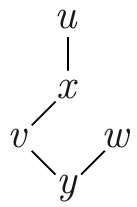}} \qquad $\graphG_4 = \;$\raisebox{-1cm}{\includegraphics[scale=1]{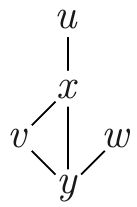}} \;\;.
\end{center}
For each one, we explain how to prove Theorem~\ref{theo:Hamiltonian2}.
\begin{description}
\item[$\graphG=\graphG_1$] The only possible conflicts are between~$x$ and a short flip with root~$v$ or~$w$. Thus, up to isomorphism of the graph, the only instance of Theorem~\ref{theo:Hamiltonian2} fitting to the configuration we are looking at is given by
\[
\shortFlip = \raisebox{-1.05cm}{\includegraphics[scale=1]{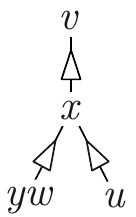}} \qquad \text{and} \qquad \shortFlip' = \raisebox{-1.05cm}{\includegraphics[scale=1]{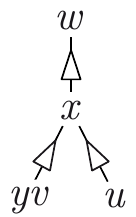}}.
\]
Observe that there exists bridges~$\bridge_v,\bridge_u,\bridge_y$ with respective roots~$\{w,v\},\{w,u\},\{w,y\}$ and a bridge~$\bridge$ with root~$\{u,x\}$. Notice that the fixed root subgraph~$\flipGraph_w(\graphG_1)$ is isomorphic to the classical (path) associahedron so that Proposition~\ref{prop:4vertices}\,\eqref{cond:bonusc} ensures that there exists a Hamiltonian cycle~$\HamiltonianCycle_w$  of the flip graph~$\flipGraph_w(\graphG_1)$ containing all the short flips~$\shortFlip,\bridge_v[w],\bridge_u[w],\bridge_y[w]$. Moreover Proposition~\ref{prop:4vertices}\,\eqref{cond:bonusa} ensures that there exists a Hamiltonian cycle~$\HamiltonianCycle_y$ (resp.~$\HamiltonianCycle_x$) of the flip graph~$\flipGraph_y(\graphG_1)$ (resp.~$\flipGraph_x(\graphG_1)$~) containing the short flip~$\bridge_y[y]$ (resp.~$\bridge{}[x]$). Proposition~\ref{prop:4vertices}\,\eqref{cond:bonusa} again gives us a Hamiltonian cycle~$\HamiltonianCycle_u$ (resp.~$\HamiltonianCycle_v$) of the flip graph~$\flipGraph_u(\graphG_1)$ (resp.~$\flipGraph_v(\graphG_1)$) containing the two short flips~$\bridge_u[u]$ and~$\bridge{}[u]$ (resp.~$\shortFlip'$ and~$\bridge_v[v]$). Note that the short flips of the bridges are all distinct since~$u,v$ and~$w$ do no disconnect the graph. Gluing all the Hamiltonian cycles of the fixed root subgraphs along the bridges as explained in Section~\ref{subsec:strategy} gives a Hamiltonian cycle of~$\flipGraph(\graphG_1)$ containing~$\shortFlip$ and~$\shortFlip'$.

\item[$\graphG=\graphG_2$] The only possible conflicts are between~$x$ and a short flip with root~$v$ or~$w$. Thus, up to isomorphism of the graph, the only instance of Theorem~\ref{theo:Hamiltonian2} fitting to the configuration we are looking at is given by
\[
\shortFlip = \raisebox{-1.05cm}{\includegraphics[scale=1]{f1}} \qquad \text{and} \qquad \shortFlip' = \raisebox{-1.05cm}{\includegraphics[scale=1]{fPrime1}}.
\]
Observe that there exists a bridge~$\bridge$ with root~$\{u,x\}$. Notice that the fixed root subgraph~$\flipGraph_w(\graphG_1)$ is isomorphic to the graph associahedron of a connected graph on~$4$ vertices so that Proposition~\ref{prop:4vertices}\,\eqref{cond:bonusd} ensures that there exists a Hamiltonian cycle~$\HamiltonianCycle_w$ of the flip graph~$\flipGraph_w(\graphG_2)$ containing the short flip~$\shortFlip$ and three short flips~$\bridge_v[w],\bridge_u[w],\bridge_y[w]$ of some bridges~$\bridge_v,\bridge_u,\bridge_y$ whose respective roots are~$\{w,v\},\{w,u\},\{w,y\}$. Moreover Proposition~\ref{prop:4vertices}\,\eqref{cond:bonusa} ensures that there exists a Hamiltonian cycle~$\HamiltonianCycle_y$ (resp.~$\HamiltonianCycle_x$) of the flip graph~$\flipGraph_y(\graphG_2)$ (resp.~$\flipGraph_x(\graphG_2)$) containing the short flip~$\bridge_y[y]$ (resp.~$\bridge{}[x]$). Proposition~\ref{prop:4vertices}\,\eqref{cond:bonusa} again gives us a Hamiltonian cycle~$\HamiltonianCycle_u$ (resp.~$\HamiltonianCycle_v$) of the flip graph~$\flipGraph_u(\graphG_2)$ (resp.~$\flipGraph_v(\graphG_2)$) containing the two short flips~$\bridge_u[u]$ and~$\bridge{}[u]$ (resp.~$\shortFlip'$ and~$\bridge_v[v]$). Note that the short flips of the bridges are all distinct since~$u,v$ and~$w$ do no disconnect the graph. Gluing all the Hamiltonian cycles of the fixed root subgraphs along the bridges as explained in Section~\ref{subsec:strategy} gives a Hamiltonian cycle of~$\flipGraph(\graphG_2)$ containing~$\shortFlip$ and~$\shortFlip'$.

\item[$\graphG=\graphG_3$] The analysis is identical to the case~$\graphG=\graphG_1$.

\item[$\graphG=\graphG_4$] The analysis is identical to the case~$\graphG=\graphG_2$.
\end{description}

We now suppose that~$\graphG$ has~$5$ vertices and that~$D=\{\{x,y\},\{x,z\}\}$. Since all edges either contain~$x$ or both~$y$ and~$z$, $\graphG$ is one of the following graphs:
\medskip
\begin{center}
$\graphG_5 = \;$\raisebox{-.6cm}{\includegraphics[scale=1]{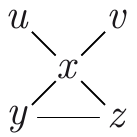}} \qquad $\graphG_6 = \;$\raisebox{-.6cm}{\includegraphics[scale=1]{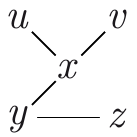}} \;\;.
\end{center}
\medskip
We note that in both of them, the only possible conflicts are between~$x$ and short flips with root either~$u$ or~$v$. Indeed, $\{x,u\}$ and~$\{x,v\}$ are the only pairs of vertices disjoint from exactly one edge, and the fixed-root subgraphs~$\flipGraph_x(\graphG_5)$ and~$\flipGraph_x(\graphG_6)$ are reduced to single flips. Using Remark~\ref{rem:ordering}, we can restrict to the cases in which~$x \notin \{v_1,v_5\}$. Again we treat separately the two graphs:

\begin{description}
\item[$\graphG=\graphG_5$] Notice that the fixed-root subgraphs~$\flipGraph_y(\graphG_5)$ and~$\flipGraph_z(\graphG_5)$ both are isomorphic to the flip graph of a star on~$4$ vertices with central vertex~$x$. So given a short flip~$\shortFlip[h]$ (resp.~$\shortFlip[k]$) with roots~$y$ (resp.~$z$), Proposition~\ref{prop:HamiltonianCycleStars} provides us with a Hamiltonian cycle~$\HamiltonianCycle_y$ (resp.~$\HamiltonianCycle_z$) of~$\flipGraph_y(\graphG_5)$ (resp.~$\flipGraph_z(\graphG_5)$) containing~$\shortFlip[h]$ (resp.~$\shortFlip[k]$) and the flip of~$\flipGraph_y(\graphG_5)$ (resp.~$\flipGraph_z(\graphG_5)$) corresponding to the long flip of~$\flipGraph(\graphG_5[\hat y])$ (resp.~$\flipGraph(\graphG_5[\hat y])$) with root~$\{x,z\}$ (resp.~$\{x,y\}$). Then gluing together the cycles~$\HamiltonianCycle_y$ and~$\HamiltonianCycle_z$ and the fixed-root subgraph~$\flipGraph_x(\graphG_5)$ as in~\fref{fig:yxzGluing} gives a tool to deal with the remaining configurations, always with the strategy of gluing Hamiltonian cycles of the fixed-root subgraphs along bridges.

\begin{figure}[h]
  \capstart
  \centerline{\includegraphics[width=.75\textwidth]{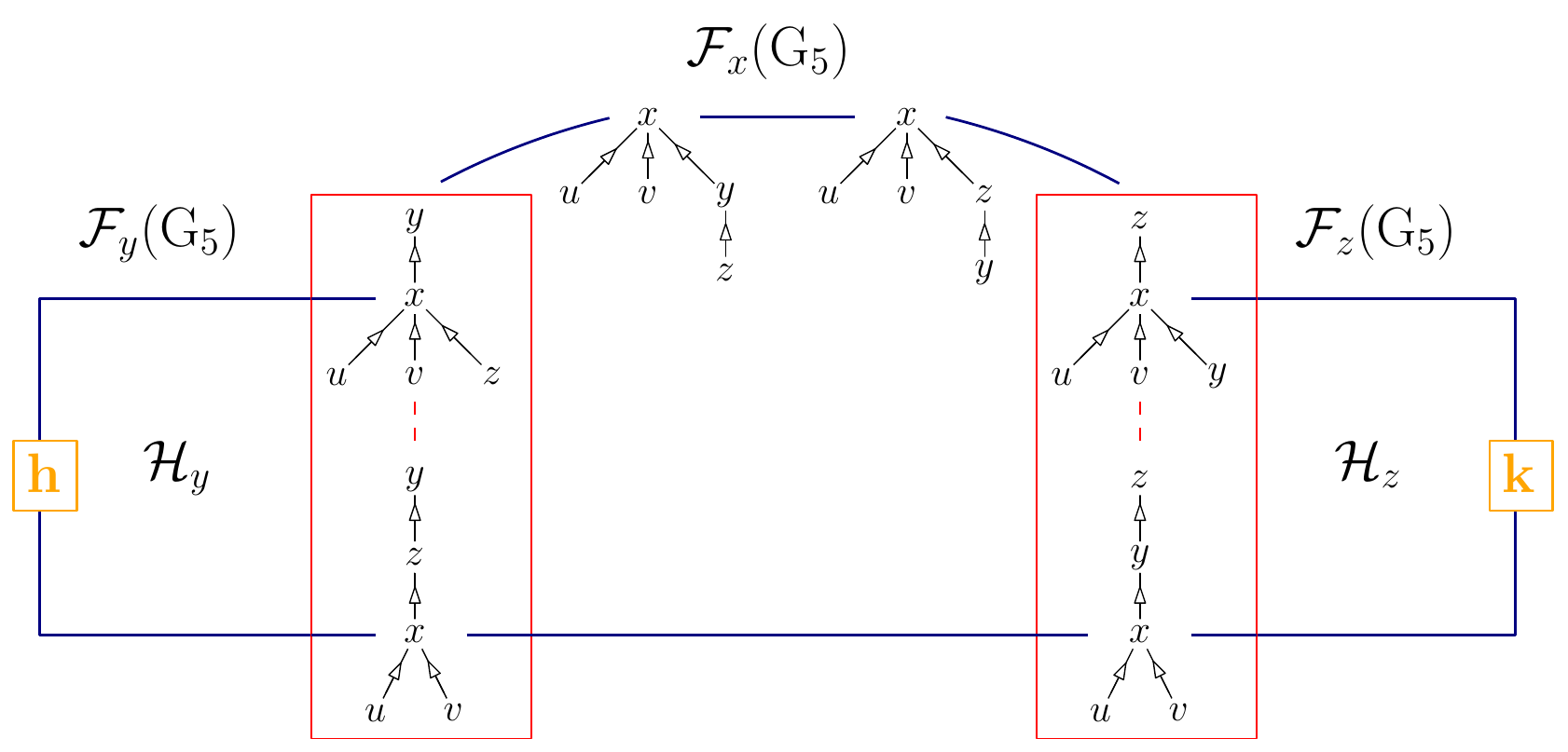}}
  \caption{How to glue together the flip graphs~$\flipGraph_y(\graphG_5),\flipGraph_x(\graphG_5)$ and~$\flipGraph_z(\graphG_5)$.}
  \label{fig:yxzGluing}
\end{figure}

\item[$\graphG=\graphG_6$] Observe that both fixed-root subgraphs~$\flipGraph_u(\graphG_6)$ and~$\flipGraph_v(\graphG_6)$ are isomorphic to the classical (path) associahedron. Thus as soon as one of the short flips~$\shortFlip$ and~$\shortFlip'$ is not in conflict with~$x$, one can find an arrangement of the vertices in the same way as when we treated the graph~$\graphG_2$ and~$\graphG_4$ (without the intermediary of the vertex~$u$) which always makes our strategy work. We thus only need to deal with the case where~$x$ is in conflict with both~$\shortFlip$ and~$\shortFlip'$, which corresponds to a single instance of Theorem~\ref{theo:Hamiltonian2}, checked by hand in \fref{fig:almostStarHamiltonian}.

\begin{figure}[h]
  \capstart
  \centerline{\includegraphics[width=\textwidth]{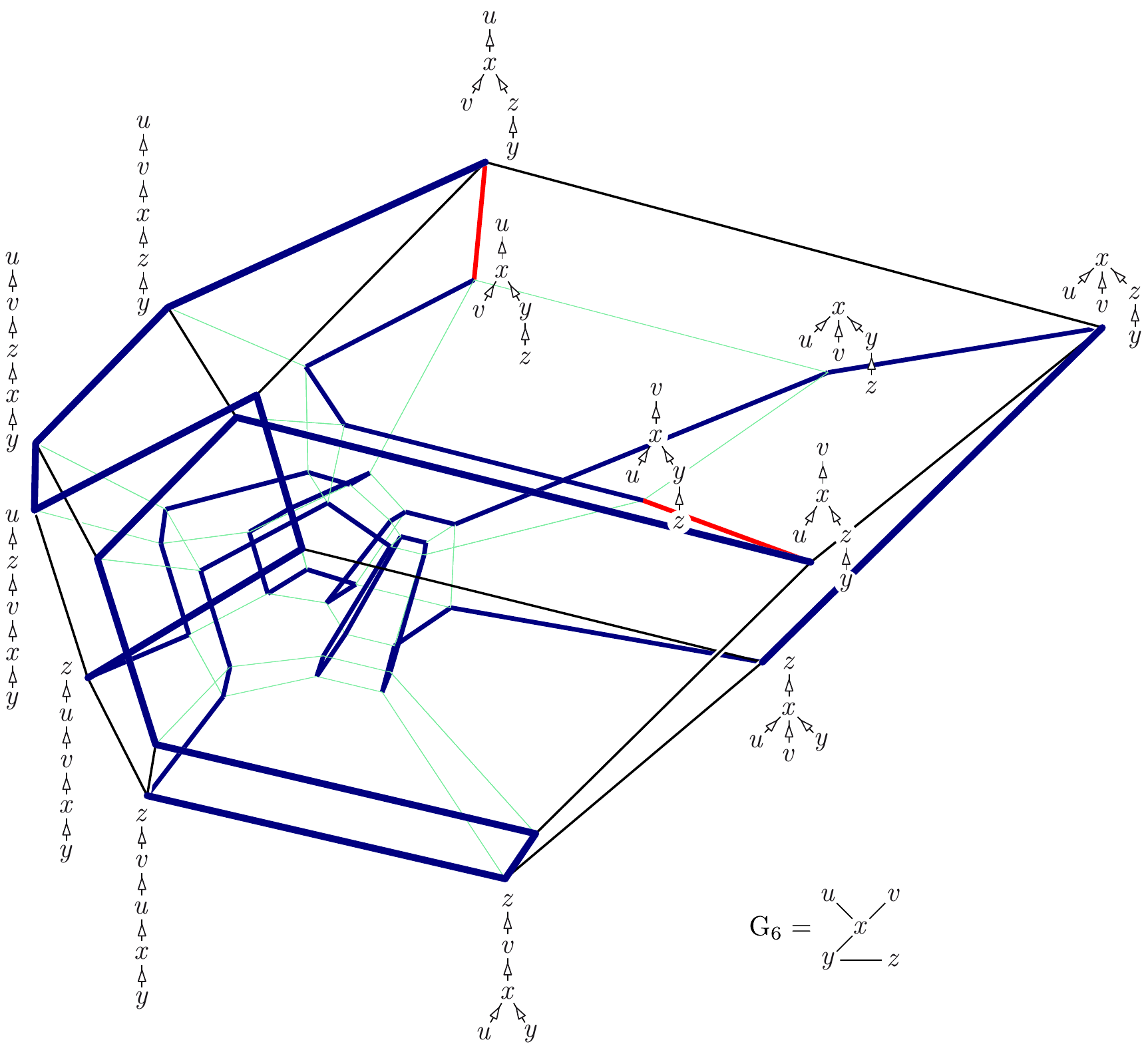}}
  \caption{The flip graph~$\flipGraph(\graphG_6)$ represented as the $1$-skeleton of the graph associahedron~$\Asso(\graphG_6)$, visualized by its Schlegel diagram. The (blue) Hamiltonian cycle passes through the only two short flips in conflict with~$x$~(in~red).}
  \label{fig:almostStarHamiltonian}
\end{figure}
\end{description}

\subsubsection{Graphs with $6$ vertices}

To finish, we need to deal with the case where~$\graphG$ has~$6$ vertices, ${D=\{\{x,y\},\{x,z\}\}}$ and~$x$ is in conflict with both~$\shortFlip$ and~$\shortFlip'$. Again~$\graphG$ can only be one of the two following graphs:
\begin{center}
$\graphG_7 = \;$\raisebox{-.6cm}{\includegraphics[scale=1]{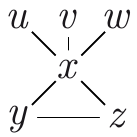}} \qquad $\graphG_8 = \;$\raisebox{-.6cm}{\includegraphics[scale=1]{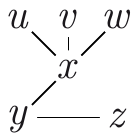}} \;\;.
\end{center}
\medskip
The graph~$\graphG_7$ is treated exactly as~$G_5$, using Remark~\ref{rem:ordering} instead of Proposition~\ref{prop:4vertices} to restrict the number of cases to analyze. In the case of $\graphG_8$, there is again a single difficult instance which can be treated by hand (since the graph associahedron~$\Asso(\graphG_8)$ has~$236$ vertices, we do not include here the resulting picture).

\section*{Acknowledgments}

We thank F.~Santos for pointing out that the diameter of the graphical zonotope is a lower bound for that of the graph associahedron, which led to the lower bound in Theorem~\ref{theo:boundsDiameter}.


\bibliographystyle{alpha}
\bibliography{graphPropertiesGraphAssociahedra}

\begin{thebibliography}{GKZ08}

\bibitem[BFS90]{BilleraFillimanSturmfels}
Louis~J. Billera, Paul Filliman, and Bernd Sturmfels.
\newblock Constructions and complexity of secondary polytopes.
\newblock {\em Adv.~Math.}, 83(2):155--179, 1990.

\bibitem[CD06]{CarrDevadoss}
Michael~P. Carr and Satyan~L. Devadoss.
\newblock Coxeter complexes and graph-associahedra.
\newblock {\em Topology Appl.}, 153(12):2155--2168, 2006.

\bibitem[CFZ02]{ChapotonFominZelevinsky}
Fr{\'e}d{\'e}ric Chapoton, Sergey Fomin, and Andrei Zelevinsky.
\newblock Polytopal realizations of generalized associahedra.
\newblock {\em Canad. Math. Bull.}, 45(4):537--566, 2002.

\bibitem[CP16]{CeballosPilaud-diameter}
C.~Ceballos and V.~Pilaud.
\newblock The diameter of type {D} associahedra and the non-leaving-face
  property.
\newblock {\em European~J.~Combin.}, 51:109--124, 2016.

\bibitem[CSZ15]{CeballosSantosZiegler}
Cesar Ceballos, Francisco Santos, and G\"unter~M. Ziegler.
\newblock Many non-equivalent realizations of the associahedron.
\newblock {\em Combinatorica}, 2015.
\newblock DOI:10.1007/s00493-014-2959-9.

\bibitem[Deh10]{Dehornoy}
Patrick Dehornoy.
\newblock On the rotation distance between binary trees.
\newblock {\em Adv. Math.}, 223(4):1316--1355, 2010.

\bibitem[Dev09]{Devadoss}
Satyan~L. Devadoss.
\newblock A realization of graph associahedra.
\newblock {\em Discrete Math.}, 309(1):271--276, 2009.

\bibitem[FS05]{FeichtnerSturmfels}
Eva~Maria Feichtner and Bernd Sturmfels.
\newblock Matroid polytopes, nested sets and {B}ergman fans.
\newblock {\em Port. Math. (N.S.)}, 62(4):437--468, 2005.

\bibitem[GKZ08]{GelfandKapranovZelevinsky}
Israel Gelfand, Mikhail M.~M. Kapranov, and Andrei Zelevinsky.
\newblock {\em Discriminants, resultants and multidimensional determinants}.
\newblock Modern Birkh\"auser Classics. Birkh\"auser Boston Inc., Boston, MA,
  2008.
\newblock Reprint of the 1994 edition.

\bibitem[HL07]{HohlwegLange}
Christophe Hohlweg and Carsten Lange.
\newblock Realizations of the associahedron and cyclohedron.
\newblock {\em Discrete Comput.~Geom.}, 37(4):517--543, 2007.

\bibitem[HLT11]{HohlwegLangeThomas}
Christophe Hohlweg, Carsten Lange, and Hugh Thomas.
\newblock Permutahedra and generalized associahedra.
\newblock {\em Adv. Math.}, 226(1):608--640, 2011.

\bibitem[HN99]{HurtadoNoy}
Ferran Hurtado and Marc Noy.
\newblock Graph of triangulations of a convex polygon and tree of
  triangulations.
\newblock {\em Comput. Geom.}, 13(3):179--188, 1999.

\bibitem[Joh63]{Johnson}
Selmer~M. Johnson.
\newblock Generation of permutations by adjacent transposition.
\newblock {\em Math. Comp.}, 17:282--285, 1963.

\bibitem[Lee89]{Lee}
Carl~W. Lee.
\newblock The associahedron and triangulations of the {$n$}-gon.
\newblock {\em European J.~Combin.}, 10(6):551--560, 1989.

\bibitem[Lod04]{Loday}
Jean-Louis Loday.
\newblock Realization of the {S}tasheff polytope.
\newblock {\em Arch.~Math.~(Basel)}, 83(3):267--278, 2004.

\bibitem[LR98]{LodayRonco}
Jean-Louis Loday and Mar{\'{\i}}a~O. Ronco.
\newblock Hopf algebra of the planar binary trees.
\newblock {\em Adv. Math.}, 139(2):293--309, 1998.

\bibitem[Luc87]{Lucas}
Joan~M. Lucas.
\newblock The rotation graph of binary trees is {H}amiltonian.
\newblock {\em J. Algorithms}, 8(4):503--535, 1987.

\bibitem[Pos09]{Postnikov}
Alexander Postnikov.
\newblock Permutohedra, associahedra, and beyond.
\newblock {\em Int. Math. Res. Not. IMRN}, (6):1026--1106, 2009.

\bibitem[Pou14]{Pournin}
Lionel Pournin.
\newblock The diameter of associahedra.
\newblock {\em Adv. Math.}, 259:13--42, 2014.

\bibitem[PS12]{PilaudSantos-brickPolytope}
Vincent Pilaud and Francisco Santos.
\newblock The brick polytope of a sorting network.
\newblock {\em European~J.~Combin.}, 33(4):632--662, 2012.

\bibitem[PS15]{PilaudStump-brickPolytope}
V.~Pilaud and C.~Stump.
\newblock Brick polytopes of spherical subword complexes and generalized
  associahedra.
\newblock {\em Adv.~Math.}, 276:1--61, 2015.

\bibitem[Sta63]{Stasheff}
Jim Stasheff.
\newblock Homotopy associativity of {H}-spaces {I}, {II}.
\newblock {\em Trans. Amer. Math. Soc.}, 108(2):293--312, 1963.

\bibitem[Ste64]{Steinhaus}
Hugo Steinhaus.
\newblock {\em One hundred problems in elementary mathematics}.
\newblock Basic Books Inc. Publishers, New York, 1964.

\bibitem[STT88]{SleatorTarjanThurston}
Daniel~D. Sleator, Robert~E. Tarjan, and William~P. Thurston.
\newblock Rotation distance, triangulations, and hyperbolic geometry.
\newblock {\em J. Amer. Math. Soc.}, 1(3):647--681, 1988.

\bibitem[Tro62]{Trotter}
H.~F. Trotter.
\newblock Algorithm 115: Perm.
\newblock {\em Commun. ACM}, 5(8):434--435, 1962.

\bibitem[Zel06]{Zelevinsky}
Andrei Zelevinsky.
\newblock Nested complexes and their polyhedral realizations.
\newblock {\em Pure Appl. Math. Q.}, 2(3):655--671, 2006.

\bibitem[Zie95]{Ziegler}
G{\"u}nter~M. Ziegler.
\newblock {\em Lectures on polytopes}, volume 152 of {\em Graduate Texts in
  Mathematics}.
\newblock Springer-Verlag, New York, 1995.

\end{thebibliography}
\label{sec:biblio}

\end{document}